\documentclass[11pt]{article}
\usepackage{amsmath}
\usepackage{amssymb,amsbsy,amsthm}
\usepackage{graphicx}
\usepackage[dvipsnames]{xcolor}
\usepackage{enumerate}
\usepackage[margin=1in]{geometry}
\usepackage{hyperref}

\allowdisplaybreaks[1]

\numberwithin{equation}{section}

\newcommand{\p}{\mathbb{P}} 
\newcommand{\E}{\mathbb{E}} 
\newcommand{\eps}{\varepsilon} 

\def\rootv{o}
\def\rooto{o}
\def\root0{o}
\def\rootrho{\rho}

\newtheorem{theorem}{Theorem}[section]
\newtheorem{lemma}[theorem]{Lemma}
\newtheorem{proposition}[theorem]{Proposition}

\newtheorem{corollary}[theorem]{Corollary}

\newtheorem{definition}[theorem]{Definition}
\newtheorem{remark}[theorem]{Remark}

\begin{document}

\title{Optimal control for diffusions on graphs}
\author{
	Laura Florescu
	\thanks{New York University; \texttt{florescu@cims.nyu.edu}.}
	\and
	Yuval Peres
	\thanks{Microsoft Research; \texttt{peres@microsoft.com}.}
	\and
	Mikl\'os Z.\ R\'acz
	\thanks{Microsoft Research; \texttt{miracz@microsoft.com}.} 
}
\date{\today}

\maketitle


\begin{abstract}
Starting from a unit mass on a vertex of a graph, we investigate the minimum number of ``\emph{controlled diffusion}'' steps needed to transport a constant mass $p$ outside of the ball of radius~$n$. 
In a step of a controlled diffusion process we may select any vertex with positive mass and topple its mass equally to its neighbors. 
Our initial motivation comes from the maximum overhang question in one dimension, but the more general case arises from optimal mass transport problems. 

On $\mathbb{Z}^{d}$ we show that $\Theta( n^{d+2} )$ steps are necessary and sufficient to transport the mass. 
We also give sharp bounds on the comb graph and $d$-ary trees. 
Furthermore, we consider graphs where simple random walk has positive speed and entropy and which satisfy Shannon's theorem, 
and show that the minimum number of controlled diffusion steps is  
$\exp{( n \cdot h / \ell ( 1 + o(1) ) )}$, 
where $h$ is the Avez asymptotic entropy and $\ell$ is the speed of random walk. 
As examples, we give precise results on Galton-Watson trees and the product of trees $\mathbb{T}_d \times \mathbb{T}_k$. 
\end{abstract}


\begin{figure}[h!]
 \centering 
 \includegraphics[width=0.43\textwidth, clip=true, trim=1.5in 1.5in 1.5in 1.5in]{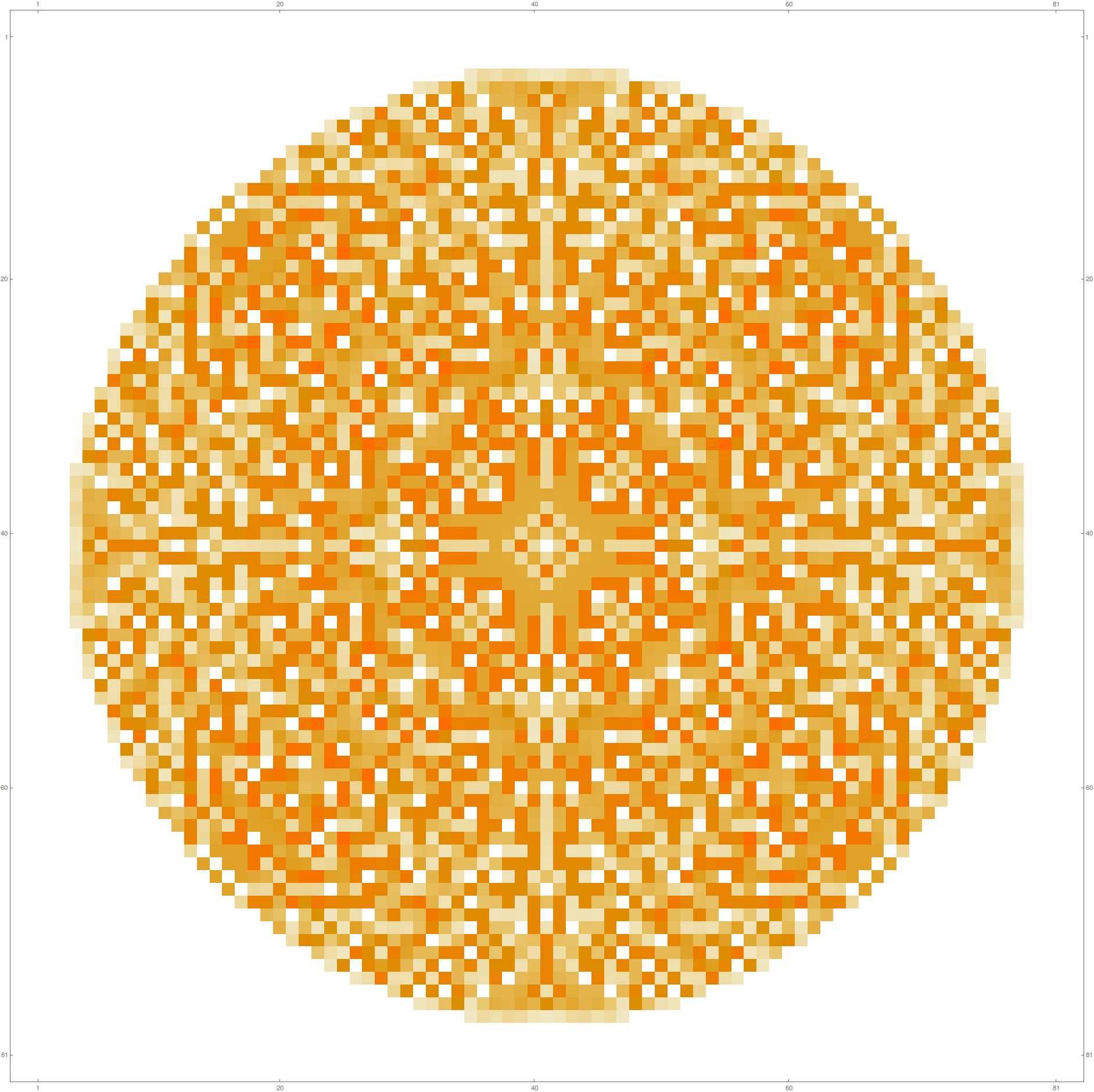}
 \caption{Starting with a unit mass at the origin of $\mathbb{Z}^{2}$, consider the process that at each step takes the vertex with largest mass and topples its mass equally to its neighbors. In case of a tie, topple all vertices with largest mass simultaneously. The figure shows the mass distribution after $10^{5}$ such greedy steps, with darker squares representing larger masses.}
 \label{fig:Z2}
\end{figure}

\section{Introduction} \label{sec:intro} 

Suppose that we have a unit mass at the origin of the $d$-dimensional lattice $\mathbb{Z}^{d}$ and we wish to move half of the mass to distance $n$. 
If the only moves we are allowed to make take a vertex and split the mass at the vertex equally among its neighbors, how many moves do we need to accomplish this goal? 
The one-dimensional case was solved by Paterson, Peres, Thorup, Winkler, and Zwick~\cite{overhang}, who studied this question due to its connections with the maximum overhang problem~\cite{paterson2006overhang,overhang}. 
The main result of this paper solves this problem in $\mathbb{Z}^{d}$ for general $d$; the proof builds on the one-dimensional case, but requires new ideas. 
We also explore this question on several other graphs, such as the comb, regular trees, Galton-Watson trees, and more.

The problem also has a probabilistic interpretation. 
Suppose there is a particle at the origin of $\mathbb{Z}^{d}$, 
as well as a controller who cannot see the particle (but who knows that the particle is initially at the origin). 
The goal of the controller is to move the particle to distance $n$ from the origin 
and it can give commands of the type ``jump if you are at vertex $v$''. 
The particle does not move unless the controller's command correctly identifies the particle's location, in which case the particle jumps to a neighboring vertex chosen uniformly at random. 
How many commands does the controller have to make in order for the particle to be at distance $n$ with probability at least $1/2$?

\subsection{Setting and main result} \label{sec:setting} 


\begin{definition}[Toppling moves]\label{def:toppling}
Given a graph $G = (V,E)$ and a mass distribution $\mu$ on the vertex set $V$, 
a \emph{toppling move} selects a vertex $v \in V$ with positive mass $\mu \left( v \right) > 0$ 
and topples (part of) the mass equally to its neighbors. 
We denote by $T_{v}^{m}$ the toppling move that topples mass $m$ at vertex $v$, 
resulting in the mass distribution $T_{v}^{m} \mu$. 

Given a subset of the vertices $A \subset V$, mass $p > 0$, and an initial mass distribution $\mu_{0}$, 
we define 
$N_{p} \left( G, A, \mu_{0} \right)$ 
to be the minimum number of toppling moves needed to move mass $p$ outside of the set $A$, 
i.e., the minimum number of toppling moves needed to obtain a mass distribution $\mu$ such that $\sum_{v \notin A} \mu \left( v \right) \geq p$. 
\end{definition}

Our interest is in the case when the initial mass distribution is a unit mass $\delta_{\rootv}$ at a given vertex $\rootv$ 
and $A$ is the (open) ball of radius $n$ around $\rootv$, 
i.e., $A = B_{n} := \left\{ u \in V : d_{G}(u,\rootv) < n \right\}$, 
where $d_{G}$ denotes graph distance in $G$. 
In other words, we wish to transport a mass of at least $p$ to distance at least $n$ away from $\rootv$. 
Our results hold for $p$ constant. 

Our main result concerns the lattice $\mathbb{Z}^{d}$: 
\begin{theorem}\label{thm:zd}
Start with initial unit mass $\delta_{\root0}$ at the origin $\root0$ of $\mathbb{Z}^{d}$, $d \geq 2$, and let $p \in (0,1)$ be constant.  
The minimum number of toppling moves needed to transport mass $p$ to distance at least $n$ from the origin is 
\[
N_{p} \left( \mathbb{Z}^{d}, B_{n}, \delta_{\root0} \right) = \Theta \left( n^{d+2} \right),  
\]
where the implied constants depend only on $d$ and $p$. 
\end{theorem}

As mentioned previously, the one-dimensional case was studied and solved in~\cite{overhang}, where the authors obtained the same result as in Theorem~\ref{thm:zd} for $d = 1$. We discuss the connection to the maximum overhang problem and related open problems in more detail at the end of the paper (see Section~\ref{sec:open}). 
Figure~\ref{fig:Z2} illustrates a greedy algorithm on $\mathbb{Z}^{2}$, which indeed transports mass $p$ to distance at least $n$ in $O \left( n^{4} \right)$ toppling moves. 


\subsection{Further results for other graphs} \label{sec:other_graphs} 

We start by giving a general upper bound on the number of toppling moves necessary to transport the mass from a vertex to outside a given set. 

\begin{theorem}\label{thm:volumeexittime}
Let $G = \left( V, E \right)$ be an infinite, connected, locally finite graph and 
let $\left\{ X_{t} \right\}_{t \geq 0}$ be simple random walk on $G$ with $X_{0} = \rootv$ for a vertex $\rootv \in V$. 
Let $A \subset V$ be a set of vertices containing $\rootv$ and 
let $T_{A}$ be the first exit time of the random walk from $A$. 
Start with initial unit mass $\delta_{\rootv}$ at $\rootv$. 
The minimum number of toppling moves needed to transport mass $p$ to outside of the set $A$ is 
\begin{equation}\label{eq:volumeexittime_bound}
N_{p} \left( G, A, \delta_{\rootv} \right) 
\leq 
\left( 1 - p \right)^{-1} \mathrm{Vol} \left( A \right) \cdot \E_{\rootv} \left[ T_{A} \right],  
\end{equation}
where $\mathrm{Vol} \left( A \right) = \left| \left\{ u \in A \right\} \right|$ denotes the volume of $A$, i.e., the number of vertices in $A$. 
\end{theorem}

In Section~\ref{sec:rw_ub} we give two proofs of this result: one using random walk on the graph to transport the mass and the other using a greedy algorithm. The two different arguments are useful because they can be extended in different ways, which, as we shall see, allows us to obtain sharper upper bounds in specific cases. 

We now consider several specific graphs, starting with the comb graph $\mathbb{C}_{2}$, 
which is obtained from $\mathbb{Z}^{2}$ by removing all horizontal edges except those on the $x$ axis; see Figure~\ref{fig:comb} for an illustration. 
\begin{theorem}\label{thm:comb}
Start with initial unit mass $\delta_{\root0}$ at the origin $\root0$ of the comb graph $\mathbb{C}_{2}$ and let $p \in (0,1)$ be constant.  
The minimum number of toppling moves needed to transport mass $p$ to distance at least $n$ from the origin is 
\[
N_{p} \left( \mathbb{C}_{2}, B_{n}, \delta_{\root0} \right) = \Theta \left( n^{7/2} \right), 
\]
where the implied constants depend only on $p$. 
\end{theorem}

Figure~\ref{fig:comb} also illustrates a greedy algorithm which achieves the upper bound in Theorem~\ref{thm:comb}. 

\begin{figure}[h!]
 \centering 
 \qquad \qquad
 \includegraphics[width=0.3\textwidth]{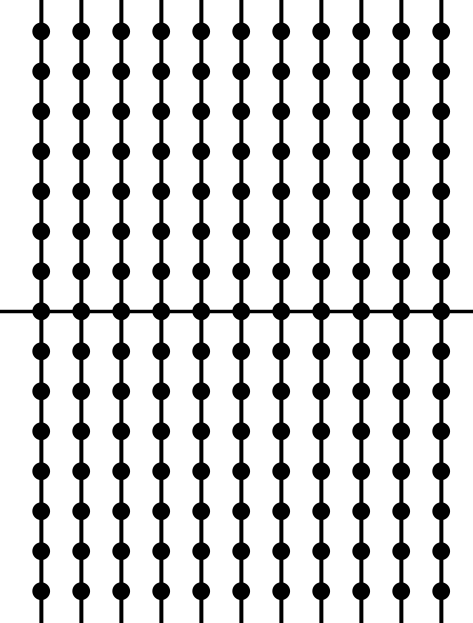} 
 \qquad 
 \includegraphics[width=0.45\textwidth, clip=true, trim=1.5in 1.5in 1.5in 1.5in]{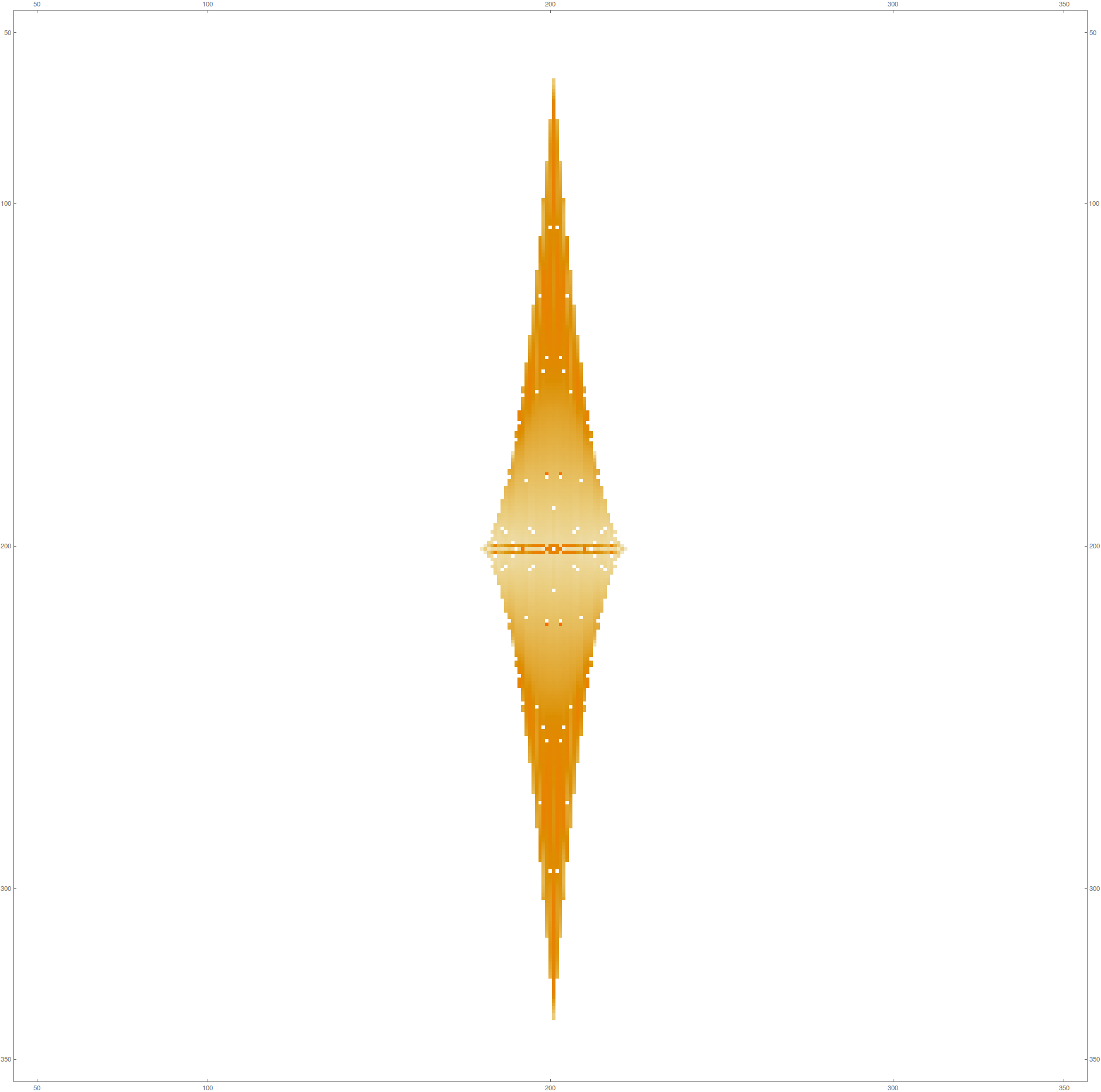}
 \caption{The comb graph $\mathbb{C}_{2}$ is on the left. On the right is the mass distribution on the comb after $10^{6}$ steps of a greedy algorithm (starting from a unit mass at the origin), with darker squares representing larger masses. The algorithm uses the same symmetric tie-breaking rule as described in the caption of Figure~\ref{fig:Z2}. The $10^{6}$ greedy steps resulted in $3439472$ toppling moves.}
 \label{fig:comb}
\end{figure}

\newpage

We also study various trees, starting with regular ones. 
\begin{theorem}\label{thm:d-ary}
Start with initial unit mass $\delta_{\rootrho}$ at the origin $\rootrho$ of the $d$-ary tree $\mathbb{T}_{d}$, $d \geq 2$, and let $p \in (0,1)$ be constant.  
The minimum number of toppling moves needed to transport mass $p$ to distance at least $n$ from the origin is 
\[
N_{p} \left( \mathbb{T}_{d}, B_{n}, \delta_{\rootrho} \right) = \Theta \left( d^{n} \right), 
\]
where the implied constants depend only on $d$ and $p$. 
\end{theorem}
We prove a general result for graphs where random walk has positive speed $\ell$ and entropy $h$ and which satisfy Shannon's theorem. 
This roughly states that 
$N_{p} \left( G, B_{n}, \delta_{\rooto} \right) = \exp \left( n \cdot \frac{h}{\ell} \cdot \left( 1 + o \left( 1 \right) \right) \right)$; 
see Section~\ref{sec:proofs_pos_speed-entropy} for a precise statement. 
This result can then be applied to specific examples, such as Galton-Watson trees and the product of two trees. 
\begin{theorem}\label{thm:GW}
Fix an offspring distribution with mean $m > 1$ and 
let $\mathrm{GWT}$ be a Galton-Watson tree obtained with this offspring distribution, on the event of nonextinction. 
Start with initial unit mass $\delta_{\rootrho}$ at the root $\rootrho$ of $\mathrm{GWT}$ and let $p \in \left(0,1\right)$ be constant. 
The minimum number of toppling moves needed to transport mass $p$ to distance at least $n$ from the origin is almost surely 
\[
N_{p} \left( \mathrm{GWT}, B_{n}, \delta_{\rootrho} \right) = \exp \left( \mathbf{dim} \cdot n \, \left( 1 + o \left( 1 \right) \right) \right),
\]
where $\mathbf{dim}$ is the dimension of harmonic measure and where the implied constants depend only on $p$ and the offspring distribution. 
\end{theorem}

When the offspring distribution is degenerate (i.e., every vertex has exactly $m$ offspring and hence the tree is the $m$-ary tree $\mathbb{T}_{m}$), then Theorem~\ref{thm:d-ary} provides a sharper result than Theorem~\ref{thm:GW}. 
However, when the offspring distribution is nondegenerate, then $\mathbf{dim} < \log m$ almost surely (see~\cite{lpp}) and hence the number of toppling moves necessary is exponentially smaller than the volume of $B_{n}$.

\begin{theorem}\label{thm:product_trees}
Let $\mathbb{T}_{d}$ denote the $(d+1)$-regular tree. 
Start with initial unit mass $\delta_{\rootrho}$ at the origin $\rootrho$ of the product of two regular trees, $\mathbb{T}_{d} \times \mathbb{T}_{k}$, and let $p \in (0,1)$ be constant.  
Assume that $d \geq k \geq 1$ and $d + k \geq 3$. 
The minimum number of toppling moves needed to transport mass $p$ to distance at least $n$ from the origin is 
\[
N_{p} \left( \mathbb{T}_{d} \times \mathbb{T}_{k}, B_{n}, \delta_{\rootrho} \right) = \theta \left( d, k \right)^{n \left( 1 + o \left( 1 \right) \right)}, 
\]
where 
$\theta \left( d, k \right) = d^{\frac{d-1}{d+k-2}} \cdot k^{\frac{k-1}{d+k-2}}$, 
and where the implied constants depend only on $d$, $k$, and $p$. 
\end{theorem}
When $d > k \geq 2$, then the volume of a ball grows as $\mathrm{Vol} \left( B_{n} \right) = \Theta \left( d^{n} \right)$, 
whereas $\theta \left( d, k \right) < d$. 
Hence the number of toppling moves necessary to transport a constant mass to distance $n$ from the root is exponentially smaller than the volume of the ball of radius $n$.

Finally, we consider graphs of bounded degree with exponential decay of the Green's function for simple random walk (see Definition~\ref{def:green}). 
\begin{theorem}\label{thm:bdd_deg_Green}
Let $G = \left( V, E \right)$ be an infinite, connected graph of bounded degree with exponential decay of the Green's function for simple random walk on $G$. 
Start with initial unit mass $\delta_{\rootv}$ at a vertex $\rootv \in V$ and let $p \in (0,1)$ be constant. 
The minimum number of toppling moves needed to transport mass $p$ to distance at least $n$ from $\rootv$ is 
\[
N_{p} \left( G, B_{n}, \delta_{\rootv} \right) = \exp \left( \Theta \left( n \right) \right),
\]
where the implied constants depend only on $p$, the maximum degree of $G$, and the exponent in the exponential bound on the Green's function. 
\end{theorem}
See Section~\ref{sec:proofs_bdd_deg_Green} where this result is restated more precisely as Theorem~\ref{thm:bdd_deg_Green_restatement} and then proved, 
and where we illustrate this result with the example of the lamplighter graph.

\subsection{Notation and preliminaries} \label{sec:notation} 

Let $G = \left( V, E \right)$ be a graph and let $\mathcal{N}_{v} := \left\{ y \in V : d_{G} \left( y, v \right) = 1 \right\}$ denote the neighborhood of a vertex $v \in V$. 
All graphs we consider in this paper are connected and locally finite (i.e., every vertex has finite degree). 
We also write $y \sim v$ for $y \in \mathcal{N}_{v}$. 
The discrete Laplacian $\Delta$ acting on functions $f : V \to \mathbb{R}$ is defined as 
\begin{equation}\label{eq:discrete_laplacian}
\Delta f \left( x \right) := \frac{1}{\left| \mathcal{N}_{x} \right|} \sum_{y \sim x} f \left( y \right) - f \left( x \right).  
\end{equation}
We can then write how a toppling move $T_{v}^{m}$ acts on a mass distribution $\mu$ as 
\begin{equation}\label{eq:toppling_move}
T_{v}^{m} \mu = \mu - m \delta_{v} + \frac{m}{\left| \mathcal{N}_{v} \right|} \sum_{y \sim v} \delta_{y} = \mu + m \Delta \delta_{v}. 
\end{equation}
We recall the well-known fact that if $G$ is a regular graph and $f$ and $g$ are two functions from $V$ to $\mathbb{R}$, with at least one of them having finite support, then 
\begin{equation}\label{eq:summation_by_parts}
\sum_{x \in V} f \left( x \right) \Delta g \left( x \right) = \sum_{x \in V} \Delta f \left( x \right) g \left( x \right), 
\end{equation}
an equality which we refer to as summation by parts. 

We also define the second moment of a mass distribution $\mu$ on $\mathbb{Z}^{d}$ as 
\begin{equation}\label{eq:second_moment}
M_{2} \left[ \mu \right] = \sum_{v \in \mathbb{Z}^{d}} \mu \left( v \right) \cdot \left\| v \right\|_{2}^{2}. 
\end{equation}

\section{Upper bound on $\mathbb{Z}^{d}$ and preliminaries for the lower bound} \label{sec:ub_and_prelims} 

We start with an upper bound on $N_{p} \left( \mathbb{Z}^{d}, B_{n}, \delta_{\root0} \right)$, stated as Theorem~\ref{thm:greedy_ub_Zd} below, which can be obtained by a greedy algorithm. 
We then introduce preliminaries for a lower bound argument which uses an appropriately defined potential. As we shall see, applying this argument directly leads to a lower bound of the correct order only in the case of $d = 1$. 
Additional ideas are required to obtain a tight lower bound for $d \geq 2$, which are then presented in Section~\ref{sec:smoothing_and_lb}.

\subsection{A greedy upper bound on $\mathbb{Z}^{d}$} \label{sec:Zd_ub} 

We use a greedy algorithm to provide an upper bound on the number of toppling moves needed to transport mass $p$ to distance $n$ from the origin in $\mathbb{Z}^{d}$.

\begin{theorem}\label{thm:greedy_ub_Zd}
Start with initial unit mass $\delta_{\root0}$ at the origin $\root0$ of $\mathbb{Z}^{d}$, $d \geq 1$. 
The minimum number of toppling moves needed to transport mass $p$ to distance at least $n$ from the origin satisfies
\begin{equation}\label{eq:Zd_greedy_ub}
N_{p} \left( \mathbb{Z}^{d}, B_{n}, \delta_{\root0} \right) < \frac{2^{d}}{\left( 1 - p \right) \times d!} n^{d+2} . 
\end{equation}
\end{theorem}

\begin{proof}
Consider the following greedy algorithm for choosing toppling moves:  
until the mass outside of $B_{n}$ is at least $p$, 
choose $v \in B_{n}$ with the largest mass in $B_{n}$ (break ties arbitrarily) 
and topple the full mass at $v$. 
Let $\mu_{0} \equiv \delta_{\root0}, \mu_{1}, \mu_{2}, \dots$ denote the resulting mass distributions, 
let $v_{i}$ denote the vertex that was toppled to get from $\mu_{i-1}$ to $\mu_{i}$, 
and let $m_{i}$ denote the mass that was toppled at this step. 
By~\eqref{eq:toppling_move} we can then write 
\begin{equation}\label{eq:mu_change}
\mu_{i} = \mu_{i-1} + m_{i} \Delta \delta_{v_{i}}. 
\end{equation}
Furthermore, let $t$ denote the number of moves necessary for this greedy algorithm to transport mass $p$ to distance at least $n$ from the origin, 
i.e., $t = \min \left\{ i \geq 1 : \mu_{i} \left( B_{n} \right) \leq 1 - p \right\}$.

We first compute how the second moment of the mass distribution changes after each toppling move. By~\eqref{eq:mu_change} we can write
\[
M_{2} \left[ \mu_{i} \right] - M_{2} \left[ \mu_{i-1} \right] 
= \sum_{x \in \mathbb{Z}^{d}} \mu_{i} \left( x \right) \left\| x \right\|_{2}^{2} - \sum_{x \in \mathbb{Z}^{d}} \mu_{i-1} \left( x \right) \left\| x \right\|_{2}^{2} 
= m_{i} \sum_{x \in \mathbb{Z}^{d}} \Delta \delta_{v_{i}} \left( x \right) \cdot \left\| x \right\|_{2}^{2}. 
\]
Now using summation by parts (see~\eqref{eq:summation_by_parts}) and the fact that $\Delta \left\| x \right\|_{2}^{2} = 1$ for every $x \in \mathbb{Z}^{d}$, we get that 
\[
\sum_{x \in \mathbb{Z}^{d}} \Delta \delta_{v_{i}} \left( x \right) \cdot \left\| x \right\|_{2}^{2} 
= \sum_{x \in \mathbb{Z}^{d}} \delta_{v_{i}} \left( x \right) \cdot \Delta \left\| x \right\|_{2}^{2} 
= \sum_{x \in \mathbb{Z}^{d}} \delta_{v_{i}} \left( x \right) 
= 1.
\]
Putting the previous two displays together we thus obtain that 
\[
M_{2} \left[ \mu_{i} \right] - M_{2} \left[ \mu_{i-1} \right] = m_{i}.
\] 
The greedy choice implies that for every $i \leq t$ we must have that 
\[
m_{i} 
\geq \frac{\mu_{i-1} \left( B_{n} \right)}{\left| B_{n} \right|} 
> \frac{1-p}{\left| B_{n} \right|}. 
\] 
This gives us the following lower bound on the second moment of $\mu_{t}$: 
\begin{equation}\label{eq:M2_greedy_lb}
M_{2} \left[ \mu_{t} \right] 
= \sum_{i=1}^{t} \left( M_{2} \left[ \mu_{i} \right] - M_{2} \left[ \mu_{i-1} \right] \right) 
= \sum_{i=1}^{t} m_{i} 
> t \times \frac{\left( 1 - p \right)}{\left| B_{n} \right|}.
\end{equation}

On the other hand, all vertices with positive mass at time $t$ have (graph) distance at most $n$ from the origin, 
and hence 
$\left\| v \right\|_{2}^{2} \leq n^{2}$ 
for every $v \in \mathbb{Z}^{d}$ 
such that $\mu_{t} \left( v \right) > 0$, 
which implies that $M_{2} \left[ \mu_{t} \right] \leq n^{2}$. 
Combining this with~\eqref{eq:M2_greedy_lb} we obtain that 
$t < \left| B_{n} \right| \times n^{2} / \left( 1 - p \right)$. 
The claim in~\eqref{eq:Zd_greedy_ub} 
then follows from the estimate 
$\left| B_{n} \right| \leq \left( 2^{d} / d! \right) n^{d}$ 
on the size of the (open) ball of radius $n$. 
\end{proof}

\begin{remark}
The greedy algorithm described in the proof above requires a tie-breaking rule, which  breaks the symmetries of $\mathbb{Z}^{d}$. 
It is also natural to consider a greedy algorithm that keeps the symmetries of $\mathbb{Z}^{d}$, such as the one described and illustrated in Figure~\ref{fig:Z2}. 
The same proof as above shows that this also transports mass $p$ to distance at least $n$ in at most 
$O \left( n^{d+2} \right)$ 
toppling moves. 
\end{remark}

\subsection{Energy of measure and potential kernel} \label{sec:energy} 

To obtain a lower bound it is natural to combine the second moment estimates with estimates for an appropriately defined potential function. We consider here a quantity called the \emph{energy} of the measure. This subsection contains the necessary definitions, together with properties of the Green's function for random walk on $\mathbb{Z}^{d}$, which are required for subsequent estimates.

\begin{definition}\label{def:energy}
The \emph{energy} of a measure $\mu$ on $\mathbb{Z}^{d}$ is defined as 
\[
 \mathcal{E}_{a} \left[ \mu \right] = \sum_{x,y \in \mathbb{Z}^{d}} a \left( x - y \right) \mu \left( x \right) \mu \left( y \right),
\]
where $a$ is the \emph{potential kernel} function. 
\end{definition} 
The energy of measure is a classical quantity; 
for more details regarding the physical context in which it arises, see, for example,~\cite{doob}. 
We will use the energy of the measure $\mu$ 
with the potential kernel function defined using the Green's function for random walk on $\mathbb{Z}^{d}$, 
which we introduce next. 

\begin{definition}[Green's function]\label{def:green}
For a random walk $\left\{ X_{k} \right\}_{k \geq 0}$ on a graph $G = \left( V, E \right)$, 
the Green's function 
$g : V \times V \to \left[ 0, \infty \right]$ 
is defined as 
\[
g \left( x, y \right) 
:= \mathbb{E}_{x} \left[ \# \left\{ k \geq 0 : X_{k} = y \right\} \right] 
= \sum_{k=0}^{\infty} \mathbb{P}_{x} \left( X_{k} = y \right) 
= \sum_{k=0}^{\infty} p^{k} \left( x, y \right),
\]
where $\mathbb{P}_{x}$ and $\mathbb{E}_{x}$ denote probabilities and expectations given that $X_{0} = x$, 
and $p^{k} \left( \cdot, \cdot \right)$ denotes the $k$-step transition probabilities. 
That is, $g \left( x, y \right)$ is the expected number of visits to $y$ by the random walk started at $x$.  
\end{definition}
Since $\mathbb{Z}^{d}$ is translation invariant, we have that 
$g \left( x, y \right) = g \left( \root0, y - x \right)$ 
for simple random walk on $\mathbb{Z}^{d}$, 
where $\root0$ denotes the origin of $\mathbb{Z}^{d}$. 
It is thus natural to define 
$g \left( x \right) := g \left( \root0, x \right)$ 
as the Green's function in $\mathbb{Z}^{d}$. 
Note that 
$g \left( x \right) = g \left( - x \right)$ 
by symmetry. 
For $d \geq 3$, simple random walk is transient in $\mathbb{Z}^{d}$ and hence $g \left( x \right)$ is finite for every $x \in \mathbb{Z}^{d}$. 
Since simple random walk is recurrent in $\mathbb{Z}$ and $\mathbb{Z}^{2}$, we have $g \left( x \right) = \infty$ for every $x \in \mathbb{Z}$ and $x \in \mathbb{Z}^{2}$. 
Thus we define instead 
\begin{equation}\label{eq:g_n}
 g_{n} \left( x \right) 
 := \mathbb{E}_{\root0} \left[ \# \left\{ k \in \left\{0, 1, \dots, n \right\} : X_{k} = x \right\} \right],
\end{equation}
the expected number of visits to $x$ until time $n$ by simple random walk started at $\root0$. 
With these notions we are ready to define the potential kernel function we will use in $\mathbb{Z}^{d}$. 
\begin{definition}[Potential kernel function for $\mathbb{Z}^{d}$]\label{def:potential}
For $d \geq 1$, define the potential kernel function 
$a : \mathbb{Z}^{d} \to \mathbb{R}$ as 
\[
 a \left( x \right) := \lim_{n \to \infty} \left\{ g_{n} \left( \root0 \right) - g_{n} \left( x \right) \right\},
\]
where $g_{n}$ is defined as in~\eqref{eq:g_n}. 
\end{definition}
This definition ensures that $a \left( x \right)$ is finite for $d = 1$ and $d = 2$ as well:  
for $d = 1$ we have that $a \left( x \right) = \left| x \right|$, and for $d = 2$ see, e.g.,~\cite[Theorem~1.6.1]{Lawler91}. 
For $d \geq 3$ we simply have that 
$a \left( x \right) = g \left( \root0 \right) - g \left( x \right)$ and we can then write the energy of a probability measure $\mu$ with this potential kernel function as 
\[
 \mathcal{E}_{a} \left[ \mu \right] = g \left( \root0 \right) - \mathcal{E}_{g} \left[ \mu \right],
\]
where
\[
 \mathcal{E}_{g} \left[ \mu \right] 
 := \sum_{x,y \in \mathbb{Z}^{d}} g \left( x, y \right) \mu \left( x \right) \mu \left( y \right).
\]

By conditioning on the first step of the random walk one can check that the discrete Laplacians of the functions $g$ and $a$ satisfy 
$\Delta g \left( x \right) = \Delta a \left( x \right) = 0$ for $x \neq \root0$, 
while at the origin $\root0$ we have that 
$\Delta g \left( \root0 \right) = -1$ 
and 
$\Delta a \left( \root0 \right) = 1$. 

We will use the following estimates for the asymptotics of the Green's function on $\mathbb{Z}^{d}$ far from the origin; 
more precise estimates are known~\cite{Lawler91,FU96}, but are not required for our purposes. 
First, when $d = 2$ then there exists an absolute constant $C_{2}$ such that 
\begin{equation}\label{eq:green_estimates_d2}
\left| a \left( x \right) - \frac{2}{\pi} \ln \left\| x \right\|_{2} - \kappa \right| 
\leq 
\frac{C_{2}}{\left\| x \right\|_{2}^{2}} 
\end{equation}
for all $x \neq \root0$, where $\kappa$ is an explicit constant whose value is not relevant for our purposes~\cite{FU96}. 
Second, for every $d \geq 3$ there exists an absolute constant $C_{d}$ such that 
\begin{equation}\label{eq:green_estimates_d3}
\left| g \left( x \right) - a_{d} \left\| x \right\|_{2}^{2-d} \right| 
\leq 
\frac{C_{d}}{\left\| x \right\|_{2}^{d-1}} 
\end{equation}
for all $x \neq \root0$, where 
$a_{d} = (d/2) \Gamma \left( d/2 - 1 \right) \pi^{-d/2} = \tfrac{2}{\left( d- 2 \right) \omega_{d}}$, 
where $\omega_{d}$ is the volume of the $L_{2}$ unit ball in $\mathbb{R}^{d}$ (see~\cite[Theorem~1.5.4]{Lawler91}).

\subsection{Comparing the energy with the second moment} \label{sec:energy_second_mom} 

To obtain a lower bound on $N_{p} \left( \mathbb{Z}^{d}, B_{n}, \delta_{\root0} \right)$ we need to compare the second moment of a mass distribution with its energy, as defined in the previous subsection. This comparison is done in the following lemma. 

\begin{lemma}\label{lem:energy_second_mom} 
Let $\mu_{0}, \mu_{1}, \dots, \mu_{t}$ be a sequence of mass distributions on $\mathbb{Z}^{d}$ resulting from toppling moves 
and let $a$ be the potential kernel function defined in Definition~\ref{def:potential}. 
Then we have that 
\begin{equation}\label{eq:energy_second_mom}
t \left( \mathcal{E}_{a} \left[ \mu_{t} \right] - \mathcal{E}_{a} \left[ \mu_{0} \right] \right) 
\geq \left( M_{2} \left[ \mu_{t} \right] - M_{2} \left[ \mu_{0} \right] \right)^{2}.
\end{equation}
\end{lemma}
\begin{proof}
For $i \in \left[ t \right]$ let $v_{i}$ denote the vertex that was toppled to get from $\mu_{i-1}$ to $\mu_{i}$, 
and let $m_{i}$ denote the mass that was toppled at this step. 
From Section~\ref{sec:Zd_ub} we know that 
$M_{2} \left[ \mu_{i} \right] - M_{2} \left[ \mu_{i-1} \right] = m_{i}$ 
for each $i \in \left[ t \right]$. 
Turning to the energy of the measure, we first recall from~\eqref{eq:mu_change} that 
$\mu_{i} = \mu_{i-1} + m_{i} \Delta \delta_{v_{i}}$ 
for every $i \in \left[ t \right]$. 
We can use this to write how the energy changes after each toppling move as follows: 
\begin{multline}
\mathcal{E}_{a} \left[ \mu_{i} \right] - \mathcal{E}_{a} \left[ \mu_{i-1} \right] 
= \sum_{x,y \in \mathbb{Z}^{d}} a \left( x - y \right) \left[ \mu_{i} \left( x \right) \mu_{i} \left( y \right) - \mu_{i-1} \left( x \right) \mu_{i-1} \left( y \right) \right] \\
\begin{aligned}
&= \sum_{x,y \in \mathbb{Z}^{d}} a \left( x - y \right) 
\left[ \left\{ \mu_{i-1} \left( x \right) + m_{i} \Delta \delta_{v_{i}} \left( x \right) \right\} \left\{ \mu_{i-1} \left( y \right) + m_{i} \Delta \delta_{v_{i}} \left( y \right) \right\} 
- \mu_{i-1} \left( x \right) \mu_{i-1} \left( y \right) \right] \\ 
&= m_{i} \sum_{x,y \in \mathbb{Z}^{d}} a \left( x - y \right) \Delta \delta_{v_{i}} \left( x \right) \mu_{i-1} \left( y \right)
+ m_{i} \sum_{x,y \in \mathbb{Z}^{d}} a \left( x - y \right) \Delta \delta_{v_{i}} \left( y \right) \mu_{i-1} \left( x \right) \\
&\quad + m_{i}^{2} \sum_{x,y \in \mathbb{Z}^{d}} a \left( x - y \right) \Delta \delta_{v_{i}} \left( x \right) \Delta \delta_{v_{i}} \left( y \right). \label{eq:three_terms}
\end{aligned}
\end{multline} 
We compute each term in the sum separately. 
Recall that 
$\Delta a \left( x \right) = \delta_{0} \left( x \right)$ 
and hence for every $y \in \mathbb{Z}^{d}$, 
$\left( \Delta a \left( \cdot - y \right) \right) \left( x \right) = \delta_{y} \left( x \right)$. 
Using summation by parts we have for every fixed $y \in \mathbb{Z}^{d}$ that 
\[
\sum_{x \in \mathbb{Z}^{d}} a \left( x - y \right) \cdot \Delta \delta_{v_{i}} \left( x \right) 
= \sum_{x \in \mathbb{Z}^{d}} \left( \Delta a \left( \cdot - y \right) \right) \left( x \right) \cdot \delta_{v_{i}} \left( x \right) 
= \delta_{v_{i}} \left( y \right). 
\]
For the first term in~\eqref{eq:three_terms} we thus have: 
\[
\sum_{x,y \in \mathbb{Z}^{d}} a \left( x - y \right) \Delta \delta_{v_{i}} \left( x \right) \mu_{i-1} \left( y \right) 
= \sum_{y \in \mathbb{Z}^{d}} \delta_{v_{i}} \left( y \right) \mu_{i-1} \left( y \right) = \mu_{i-1} \left( v_{i} \right). 
\]
Since $a \left( x - y \right) = a \left( y - x \right)$ 
we have that the second term in~\eqref{eq:three_terms} is equal to the first one. 
Finally we can compute the third term similarly: 
\[
\sum_{x,y \in \mathbb{Z}^{d}} a \left( x - y \right) \Delta \delta_{v_{i}} \left( x \right) \Delta \delta_{v_{i}} \left( y \right) 
= \sum_{y \in \mathbb{Z}^{d}} \delta_{v_{i}} \left( y \right) \Delta \delta_{v_{i}} \left( y \right) 
= \Delta \delta_{v_{i}} \left( v_{i} \right) = -1. 
\]
Putting together the previous two displays with~\eqref{eq:three_terms}, we can conclude that 
\[
\mathcal{E}_{a} \left[ \mu_{i} \right] - \mathcal{E}_{a} \left[ \mu_{i-1} \right] 
= 2 m_{i} \mu_{i-1} \left( v_{i} \right) - m_{i}^{2} 
\geq m_{i}^{2},
\]
where the last step follows because $m_{i} \leq \mu_{i-1} \left( v_{i} \right)$.

The claimed inequality~\eqref{eq:energy_second_mom} now follows by the Cauchy-Schwarz inequality: 
\[
\left( M_{2} \left[ \mu_{t} \right] - M_{2} \left[ \mu_{0} \right] \right)^{2} 
= \left( \sum_{i=1}^{t} m_{i} \right)^{2} 
\leq t \sum_{i=1}^{t} m_{i}^{2} 
\leq t \left( \mathcal{E}_{a} \left[ \mu_{t} \right] - \mathcal{E}_{a} \left[ \mu_{0} \right] \right). \qedhere
\]
\end{proof}

\subsection{An initial lower bound argument} \label{sec:initial_lb} 

A lower bound of the correct order in dimension $d=1$ now follows (see also~\cite{overhang} where this argument first appeared). 
Suppose that a sequence of $t$ toppling moves are applied to obtain mass distributions 
$\mu_{0} \equiv \delta_{\root0}, \mu_{1}, \mu_{2}, \dots, \mu_{t}$ 
that satisfy $\mu_{i} = T_{v_{i}}^{m_{i}} \mu_{i-1}$ for every $i \in \left[ t \right]$ and $\mu_{t} \left( B_{n} \right) \leq 1 - p$. 
We may assume that $\left| v_{i} \right| \leq n-1$ for every $i \leq t$; 
any other toppling move can be removed from the sequence to obtain a shorter sequence that still moves mass $p$ to distance $n$ from the origin.

Now recall that $a \left( x \right) = \left| x \right|$ when $d = 1$. Since 
$\mu_{t} \left( \left\{ v : \left| v \right| > n \right\} \right) = 0$, 
we have that 
$\mathcal{E}_{a} \left[ \mu_{t} \right] \leq 2n$. 
On the other hand, since 
$\mu_{t} \left( \left\{ v : \left| v \right| \geq n \right\} \right) \geq p$, 
we have that 
$M_{2} \left[ \mu_{t} \right] \geq p n^{2}$. 
By Lemma~\ref{lem:energy_second_mom} we thus have that 
$t \times 2n \geq \left( p n^{2} \right)^{2}$, 
implying that 
\[
N_{p} \left( \mathbb{Z}, B_{n}, \delta_{\root0} \right) \geq \frac{p^{2}}{2} n^{3},
\]
which matches the upper bound of Theorem~\ref{thm:greedy_ub_Zd} up to constant factors in $p$.

However, the same argument for $d \geq 2$ (using the estimates for the Green's function from~\eqref{eq:green_estimates_d2} and~\eqref{eq:green_estimates_d3}; we leave the details to the reader) only provides the following estimates: 
there exists a constant $C$ depending only on 
$d$ and 
$p$ such that 
\begin{equation*}
N_{p} \left( \mathbb{Z}^{2}, B_{n}, \delta_{\root0} \right) \geq \frac{C n^{4}}{\log \left( n \right)}, 
\end{equation*}
and 
\begin{equation*}
N_{p} \left( \mathbb{Z}^{d}, B_{n}, \delta_{\root0} \right) \geq C n^{4}, 
\end{equation*}
for $d \geq 3$. 
Therefore, to obtain a tight lower bound in dimensions $d \geq 2$, a new idea is needed. The idea, presented in the following section, is to perform an initial smoothing of the mass distribution.

\section{Smoothing and the lower bound on $\mathbb{Z}^{d}$} \label{sec:smoothing_and_lb} 

The previous section provides the basis for the proof of Theorem~\ref{thm:zd}, but applying the arguments directly leads to a suboptimal lower bound, as described in Section~\ref{sec:initial_lb}. 
The remedy is to perform an initial smoothing of the mass distribution. 
In this section we first describe the smoothing operation in general in Section~\ref{sec:smoothing}, followed by describing the specifics of smoothing in $\mathbb{Z}^{d}$ in Section~\ref{sec:smoothing_zd}. 
We conclude with the proof of the lower bound on $\mathbb{Z}^{d}$ in Section~\ref{sec:Zd_lb}.

\subsection{Smoothing of distributions} \label{sec:smoothing} 

For the proofs of the lower bounds on most families of graphs investigated in this paper we use a certain smoothing of the mass distribution. 
That is, we first perform some toppling moves to obtain a mass distribution $\widetilde{\mu}$ that is ``smooth'' in the sense that it is approximately uniform over a subset of the ball $B_{n}$. 
In this subsection we show that it is valid to use smoothing for lower bound arguments, since the minimum number of toppling moves necessary to transport mass $p$ outside of a set $A$ cannot increase by smoothing. 
What then remains to be estimated (for each family of graphs separately) 
is the minimum number of toppling moves necessary to transport mass $p$ to distance $n$ started from the smooth distribution $\widetilde{\mu}$. 

\begin{lemma}[Smoothing weakly reduces the minimum number of toppling moves]\label{lem:smoothing}
Start with mass distribution $\mu$ on a graph $G$, 
and let $A \subseteq V \left( G \right)$. 
Suppose that toppling mass $m$ at vertex $v \in A$ is a valid toppling move. 
We then have that 
\begin{equation}\label{eq:smoothing}
N_{p} \left( G, A, T_{v}^{m} \mu \right) \leq N_{p} \left( G, A, \mu \right).
\end{equation}
\end{lemma}
\begin{proof}
We prove the statement by induction on $t := N_{p} \left( G, A, \mu \right)$. 
For the base case of $t = 0$, 
if $N_{p} \left( G, A, \mu \right) = 0$, then $\mu \left( A \right) \leq 1 - p$. 
Since $v \in A$, no mass can enter $A$ from outside of $A$ in the toppling move, 
so $T_{v}^{m} \mu \left( A \right) \leq 1 - p$ and hence 
$N_{p} \left( G, A, T_{v}^{m} \mu \right) = 0$. 

For the induction step, 
let $t = N_{p} \left( G, A, \mu \right)$ and 
let $\mu \equiv \mu_{0}, \mu_{1}, \dots, \mu_{t}$ be a series of mass distributions 
such that $\mu_{i}$ is obtained from $\mu_{i-1}$ by a toppling move at vertex $v_{i}$ 
with mass $m_{i}$ being toppled, 
i.e., $\mu_{i} = T_{v_{i}}^{m_{i}} \mu_{i-1}$, 
and such that $\mu_{t} \left( A \right) \leq 1 - p$. 
Due to the optimality of the sequence of toppling moves we have that $N_{p} \left( G, A, \mu_{1} \right) = t - 1$.

Consider first the case that $v \neq v_{1}$. 
In this case the toppling moves $T_{v}^{m}$ and $T_{v_{1}}^{m_{1}}$ commute, i.e., 
$T_{v_{1}}^{m_{1}} T_{v}^{m} \mu = T_{v}^{m} T_{v_{1}}^{m_{1}} \mu$. 
Hence 
\[
N_{p} \left( G, A, T_{v}^{m} \mu \right) 
\leq N_{p} \left( G, A, T_{v_{1}}^{m_{1}} T_{v}^{m} \mu \right) + 1 
= N_{p} \left( G, A, T_{v}^{m} \mu_{1} \right) + 1 
\leq N_{p} \left( G, A, \mu_{1} \right) + 1 = t,
\]
where the second inequality is due to the induction hypothesis. 

Now consider the case that $v = v_{1}$. If $m_{1} > m$, then 
\[
N_{p} \left( G, A, T_{v}^{m} \mu \right) 
\leq N_{p} \left( G, A, T_{v_{1}}^{m_{1} - m} T_{v}^{m} \mu \right) + 1 
= N_{p} \left( G, A, \mu_{1} \right) + 1 
= t.
\]
If $m \geq m_{1}$, then 
\[
N_{p} \left( G, A, T_{v}^{m} \mu \right) 
= N_{p} \left( G, A, T_{v}^{m - m_{1}} \mu_{1} \right) 
\leq N_{p} \left( G, A, \mu_{1} \right) 
= t - 1,
\]
where the inequality is again due to the induction hypothesis. 
\end{proof}

Iterating this lemma we immediately obtain the following corollary. 
\begin{corollary}\label{cor:smoothing}
Let $\mu_{0}, \mu_{1}, \dots, \mu_{t}$ be a sequence of mass distributions on a graph $G$ 
such that for every $i \in \left[ t \right]$, the mass distribution $\mu_{i}$ is obtained from $\mu_{i-1}$ by applying a toppling move at vertex $v_{i} \in V \left( G \right)$, toppling a mass $m_{i}$, 
i.e., $\mu_{i} = T_{v_{i}}^{m_{i}} \mu_{i-1}$. 
Let $A \subseteq V \left( G \right)$ and assume that $v_{i} \in A$ for every $i \in \left[ t \right]$. 
Then we have that 
\[
N_{p} \left( G, A, \mu_{t} \right) \leq N_{p} \left( G, A, \mu_{0} \right).
\]
\end{corollary}

Another corollary of the lemma above is that we can assume without loss of generality that at every move we topple \emph{all} the mass at a given vertex. 
Given a graph $G = (V, E)$, a subset of the vertices $A \subset V$, mass $p > 0$, and an initial mass distribution $\mu_{0}$, 
we define 
$N_{p}^{\mathrm{full}} \left( G, A, \mu_{0} \right)$ 
to be the minimum number of toppling moves needed to move mass $p$ outside of the set $A$, 
where at every toppling move we have to topple all the mass at a given vertex. 
\begin{corollary}\label{cor:full_toppling}
We have that 
$N_{p} \left( G, A, \mu_{0} \right) = N_{p}^{\mathrm{full}} \left( G, A, \mu_{0} \right)$. 
\end{corollary}
\begin{proof}
Since allowing only full topplings is more restrictive than allowing partial topplings, we have that 
$N_{p} \left( G, A, \mu_{0} \right) \leq N_{p}^{\mathrm{full}} \left( G, A, \mu_{0} \right)$. 
We prove the other inequality, i.e., that 
$N_{p} \left( G, A, \mu_{0} \right) \geq N_{p}^{\mathrm{full}} \left( G, A, \mu_{0} \right)$, 
by induction on $t := N_{p} \left( G, A, \mu_{0} \right)$. 
For the base of $t = 0$, 
if $N_{p} \left( G, A, \mu_{0} \right) = 0$, then $\mu_{0} \left( A \right) \leq 1 - p$, 
and hence 
$N_{p}^{\mathrm{full}} \left( G, A, \mu_{0} \right) = 0$.

For the induction step, 
let $t = N_{p} \left( G, A, \mu_{0} \right)$ and 
let $\mu_{0}, \mu_{1}, \dots, \mu_{t}$ be a series of mass distributions 
such that $\mu_{i}$ is obtained from $\mu_{i-1}$ by a toppling move at vertex $v_{i}$ 
with mass $m_{i}$ being toppled, 
i.e., $\mu_{i} = T_{v_{i}}^{m_{i}} \mu_{i-1}$, 
and such that $\mu_{t} \left( A \right) \leq 1 - p$. 
Due to the optimality of the sequence of toppling moves we have that $N_{p} \left( G, A, \mu_{1} \right) = t - 1$. 
Define the mass distribution
$\mu_{1}' = T_{v_{1}}^{\mu_{0} \left( v_{1} \right)} \mu_{0}$,  
which corresponds to toppling all the original mass at $v_{1}$, 
and note that 
$\mu_{1}' = T_{v_{1}}^{\mu_{0} \left( v_{1} \right) - m_{1}} \mu_{1}$. 
By Lemma~\ref{lem:smoothing} we have that 
$N_{p} \left( G, A, \mu_{1}' \right) \leq N_{p} \left( G, A, \mu_{1} \right)$ 
and by the induction hypothesis we have that 
$N_{p} \left( G, A, \mu_{1}' \right) = N_{p}^{\mathrm{full}} \left( G, A, \mu_{1}' \right)$. 
Therefore we obtain that 
\[
N_{p}^{\mathrm{full}} \left( G, A, \mu_{0} \right) 
\leq 
N_{p}^{\mathrm{full}} \left( G, A, \mu_{1}' \right) + 1
= 
N_{p} \left( G, A, \mu_{1}' \right) + 1 
\leq t. \qedhere 
\]
\end{proof}

\subsection{Smoothing in $\mathbb{Z}^{d}$} \label{sec:smoothing_zd} 

For the initial smoothing in $\mathbb{Z}^{d}$ we leverage connections between our controlled diffusion setting and the \emph{divisible sandpile model}, and use results by Levine and Peres~\cite{lp} on this model. 
In the divisible sandpile each site $x \in \mathbb{Z}^{d}$ starts with mass $\nu_{0} \left( x \right) \in \mathbb{R}_{\geq 0}$. 
A site $x$ is \emph{full} if its mass is at least $1$. 
A \emph{divisible sandpile move} at $x$, denoted by $\mathcal{D}_{x}$, consists of no action if $x$ is not full, 
and consists of keeping mass $1$ at $x$ and splitting any excess mass equally among its neighbors if $x$ is full. 

Recall from~\eqref{eq:toppling_move} that the mass distribution after a toppling move can be written as $T_{v}^{m} \mu = \mu + m \Delta \delta_{v}$. 
Similarly, for a mass distribution $\mu$ and a site $x \in \mathbb{Z}^{d}$, 
the mass distribution after a divisible sandpile move at $x$ can be written as 
\begin{equation}\label{eq:mu_change_ds}
\mathcal{D}_{x} \mu = \mu + \max \left\{ \mu \left( x \right) - 1, 0 \right\} \Delta \delta_{x}.
\end{equation}

Note that individual divisible sandpile moves do not commute; however, the divisible sandpile is ``abelian'' in the following sense. 
\begin{proposition}[Levine and Peres~\cite{lp}]\label{prop:ds_abelian}
Let $x_{1}, x_{2}, \dots \in \mathbb{Z}^{d}$ be a sequence with the property that for any $x \in \mathbb{Z}^{d}$ there are infinitely many terms $x_{k} = x$. 
Let $\nu_{0}$ denote the initial mass distribution and assume that $\nu_{0}$ has finite support. 
Let 
\begin{align*}
u_{k} \left( x \right) &= \text{ \emph{total mass emitted by} } x \text{ \emph{after divisible sandpile moves} } x_{1}, \dots, x_{k}; \\
\nu_{k} \left( x \right) &= \text{ \emph{amount of mass present at} } x \text{ \emph{after divisible sandpile moves} } x_{1}, \dots, x_{k}. 
\end{align*}
Then $u_{k} \uparrow u$ and $\nu_{k} \to \nu \leq 1$. 
Moreover, the limits $u$ and $\nu$ are independent of the sequence $\left\{ x_{k} \right\}$. 
\end{proposition}

The limit $\nu$ represents the final mass distribution and sites $x \in \mathbb{Z}^{d}$ with $\nu \left( x \right) = 1$ are called \emph{fully occupied}. 
We are interested primarily in the case when the initial mass distribution is a point mass at the origin: $\nu_{\root0} = m \delta_{\root0}$ for some $m > 0$. 
The natural question then is to identify the shape of the resulting domain $D_{m}$ of fully occupied sites. 
The following result states that $D_{m}$ is very close to a Euclidean ball. 
Since in this paper the notation $B_{n}$ is reserved for the $L_{1}$ ball (and the graph distance ball more generally), 
we denote by $B_{r}^{(2)} = \left\{ x \in \mathbb{Z}^{d} : \left\| x \right\|_{2} < r \right\}$ the (open) $L_{2}$ ball around the origin. 

\begin{theorem}[Levine and Peres~\cite{lp}]\label{thm:ds_limit}
For $m \geq 0$ let $D_{m} \subset \mathbb{Z}^{d}$ be the domain of fully occupied sites for the divisible sandpile formed from a pile of mass $m$ at the origin. 
There exist constants $c$ and $c'$ depending only on $d$ such that 
\[
B_{r - c}^{(2)} \subset D_{m} \subset B_{r + c'}^{(2)},
\]
where $r = \left( m / \omega_{d} \right)^{1/d}$ and $\omega_{d}$ is the volume of the $L_{2}$ unit ball in $\mathbb{R}^{d}$.
\end{theorem}

We note that the sequence of divisible sandpile moves started from a pile of mass $m$ at the origin could potentially be infinite. 
However, there exists a finite $K$ such that $\nu_{K} \left( x \right) \leq 1 + \eps$ for every $x \in \mathbb{Z}^{d}$ and for some small $\eps > 0$. 
This is useful for proving the following corollary of the theorem above.

\begin{corollary}\label{cor:divisible_sandpile}
For every $c \in \left( 0, 1 \right)$ 
there exists a finite sequence of toppling moves 
that takes the mass distribution $\delta_{\root0}$ on $\mathbb{Z}^{d}$ 
to a mass distribution $\mu$ on $\mathbb{Z}^{d}$ 
for which the following two properties hold:
\begin{align}
\forall x \in B_{cn}^{(2)} \, : &\quad \mu \left( x \right) \leq \frac{2}{\mathrm{Vol} \left( B_{cn}^{(2)} \right)}, \label{eq:x_in_B} \\
\forall x \notin B_{cn}^{(2)} \, : &\quad \mu \left( x \right) = 0,  \label{eq:x_notin_B}
\end{align}
where $\mathrm{Vol} \left( B_{cn}^{(2)} \right) = \left| \left\{ x \in \mathbb{Z}^{d} : \left\| x \right\|_{2} < cn \right\} \right|$ denotes the volume of the ball $B_{cn}^{(2)}$. 
\end{corollary}

\begin{proof} The result follows from Theorem~\ref{thm:ds_limit} by scaling the masses by $m$, for both the mass distributions and the divisible sandpile moves. 
\end{proof}

\subsection{A lower bound on $\mathbb{Z}^{d}$} \label{sec:Zd_lb} 

We are now ready to prove Theorem~\ref{thm:zd}. 
By performing an initial smoothing as detailed in Section~\ref{sec:smoothing_zd}, we are able to obtain a lower bound that matches the upper bound of Theorem~\ref{thm:greedy_ub_Zd} up to constant factors. 

\begin{theorem}\label{thm:smoothing_lb_Zd} 
Start with initial unit mass $\delta_{\root0}$ at the origin $\root0$ of $\mathbb{Z}^{d}$, $d \geq 2$. 
There exists a constant $C$ depending only on $d$ and $p$ such that 
the minimum number of toppling moves needed to transport mass $p$ to distance at least $n$ from the origin satisfies
\begin{equation}\label{eq:Zd_smoothing_lb}
N_{p} \left( \mathbb{Z}^{d}, B_{n}, \delta_{\root0} \right) \geq C n^{d+2}. 
\end{equation}
\end{theorem}

\begin{proof}
The first step is to smooth the distribution $\delta_{\root0}$. 
Let $c := \sqrt{p / (2d)}$. 
By Corollary~\ref{cor:divisible_sandpile} there exists a finite sequence of toppling moves taking $\delta_{\root0}$ to a mass distribution $\mu$ satisfying~\eqref{eq:x_in_B} and~\eqref{eq:x_notin_B}. 
By Corollary~\ref{cor:smoothing} we have that 
$N_{p} \left( \mathbb{Z}^{d}, B_{n}, \delta_{\root0} \right) \geq N_{p} \left( \mathbb{Z}^{d}, B_{n}, \mu \right)$, 
so it suffices to bound $N_{p} \left( \mathbb{Z}^{d}, B_{n}, \mu \right)$ from below. 

Suppose that starting from $\mu$ a sequence of $t$ toppling moves are applied to obtain mass distributions 
$\mu_{0} \equiv \mu, \mu_{1}, \mu_{2}, \dots, \mu_{t}$ that satisfy  
$\mu_{t} \left( B_{n} \right) \leq 1 - p$. 
Let $v_{i}$ denote the vertex that was toppled to get from $\mu_{i-1}$ to $\mu_{i}$, 
and let $m_{i}$ denote the mass that was toppled at this step. 
We may assume that $\left\| v_{i} \right\|_{1} \leq n-1$ for every $i \leq t$, since  
any other toppling move can be removed from the sequence to obtain a shorter sequence that still moves mass $p$ to distance $n$ from the origin. 

By Lemma~\ref{lem:energy_second_mom} we have that 
\begin{equation}\label{eq:t_lb}
 t \geq \frac{\left( M_{2} \left[ \mu_{t} \right] - M_{2} \left[ \mu_{0} \right] \right)^{2}}{\mathcal{E}_{a} \left[ \mu_{t} \right] - \mathcal{E}_{a} \left[ \mu_{0} \right]}
\end{equation}
and in the following we bound the numerator and the denominator separately, starting with the numerator. 

Since $\mu_{t} \left( \left\{ x \in \mathbb{Z}^{d} : \left\| x \right\|_{1} \geq n \right\} \right) \geq p$ 
and $\left\| x \right\|_{2} \geq \left\| x \right\|_{1} / \sqrt{d}$, 
we have that 
$M_{2} \left[ \mu_{t} \right] \geq \tfrac{pn^{2}}{d}$. 
On the other hand, 
the support of $\mu_{0}$ is contained within $B_{cn}^{(2)}$ and so 
$M_{2} \left[ \mu_{0} \right] \leq c^{2} n^{2} = \tfrac{pn^{2}}{2d}$. 
Putting these two estimates together we obtain that 
\begin{equation}\label{eq:second_mom_estimate}
\left( M_{2} \left[ \mu_{t} \right] - M_{2} \left[ \mu_{0} \right] \right)^{2} \geq \frac{p^{2}}{4 d^{2}} n^{4}.
\end{equation}

From~\eqref{eq:t_lb} and~\eqref{eq:second_mom_estimate} we have that in order to show~\eqref{eq:Zd_smoothing_lb}, what remains is to show that 
\begin{equation}\label{eq:energy_estimate}
\mathcal{E}_{a} \left[ \mu_{t} \right] - \mathcal{E}_{a} \left[ \mu_{0} \right] \leq C' n^{2-d}
\end{equation}
for some constant $C'$ depending only on $d$ and $p$. 
At this point the proof slightly differs for $d = 2$ and $d \geq 3$. We start with the case of $d \geq 3$. 

Recall from Section~\ref{sec:energy} that when $d \geq 3$ then 
\[
\mathcal{E}_{a} \left[ \mu_{t} \right] - \mathcal{E}_{a} \left[ \mu_{0} \right] 
= \mathcal{E}_{g} \left[ \mu_{0} \right] - \mathcal{E}_{g} \left[ \mu_{t} \right] 
\leq \mathcal{E}_{g} \left[ \mu_{0} \right]. 
\]
We estimate this latter quantity by dividing $\mathbb{Z}^{d} \times \mathbb{Z}^{d}$ into shells 
\[
E_{k} := \left\{ \left( x, y \right) \in \mathbb{Z}^{d} \times \mathbb{Z}^{d} : \frac{2n}{2^{k}} < \left\| x - y \right\|_{2} \leq \frac{2n}{2^{k-1}} \right\}
\]
and estimating the sum on each shell separately. Since the support of $\mu_{0}$ is contained in $B_{cn}^{(2)}$, we can write 
\begin{equation}\label{eq:sum_shells}
\mathcal{E}_{g} \left[ \mu_{0} \right] 
= \sum_{x \in B_{cn}^{(2)}} g \left( \root0 \right) \mu_{0} \left( x \right)^{2} 
+ \sum_{k=1}^{K} \sum_{\left( x, y \right) \in E_{k}} g \left( x - y \right) \mu_{0} \left( x \right) \mu_{0} \left( y \right),
\end{equation}
where $K = \left\lceil \log_{2} \left( 2n \right) \right\rceil$. 
Using~\eqref{eq:x_in_B} we have that the first term in~\eqref{eq:sum_shells} can be bounded as follows: 
\begin{equation}\label{eq:xy_equal}
\sum_{x \in B_{cn}^{(2)}} g \left( \root0 \right) \mu_{0} \left( x \right)^{2} 
\leq \sum_{x \in B_{cn}^{(2)}} \frac{4 g \left( \root0 \right)}{\mathrm{Vol} \left( B_{cn}^{(2)} \right)^{2}} 
= \frac{4 g \left( \root0 \right)}{\mathrm{Vol} \left( B_{cn}^{(2)} \right)} 
= O \left( n^{-d} \right),
\end{equation}
where in the last estimate we used that $\mathrm{Vol} \left( B_{cn}^{(2)} \right) = \Theta \left( n^{d} \right)$. 
Now if $x \neq y$ then we have from~\eqref{eq:green_estimates_d3} that 
\[
g \left( x - y \right) \leq a_{d} \left\| x - y \right\|_{2}^{2-d} + C_{d} \left\| x - y \right\|_{2}^{1-d}
\]
and so if $\left( x, y \right) \in E_{k}$ then 
\begin{equation}\label{eq:g_estimate_shell}
g \left( x - y \right) 
\leq a_{d} \left( \frac{2n}{2^{k-1}} \right)^{2-d} + C_{d} \left( \frac{2n}{2^{k-1}} \right)^{1-d} 
= O \left( n^{2-d} \times 2^{dk} \times  \left( 2^{-2k} + n^{-1} \times 2^{-k} \right) \right).
\end{equation}
Now to bound the mass of a shell first note that for any $x \in \mathbb{Z}^{d}$ we have that 
\[
\sum_{y : \left( x, y \right) \in E_{k}} \mu_{0} \left( y \right) 
\leq \mathrm{Vol} \left( B_{\frac{2n}{2^{k-1}}}^{(2)} \right) \times \frac{2}{\mathrm{Vol} \left( B_{cn}^{(2)} \right)}  
= O \left( 2^{-dk} \right),
\]
where we used again that 
$\mathrm{Vol} \left( B_{r}^{(2)} \right) = \Theta \left( r^{d} \right)$. 
This then implies that 
\begin{equation}\label{eq:mass_of_shell}
\sum_{\left( x, y \right) \in E_{k}} \mu_{0} \left( x \right) \mu_{0} \left( y \right) 
= O \left( 2^{-dk} \right). 
\end{equation}
Putting together~\eqref{eq:g_estimate_shell} and~\eqref{eq:mass_of_shell} we obtain that 
\[
\sum_{\left( x, y \right) \in E_{k}} g \left( x - y \right) \mu_{0} \left( x \right) \mu_{0} \left( y \right) 
= O \left( n^{2-d} \times  \left( 2^{-2k} + n^{-1} \times 2^{-k} \right) \right). 
\]
Summing this over $k$ we get that 
\[
 \sum_{k=1}^{K} \sum_{\left( x, y \right) \in E_{k}} g \left( x - y \right) \mu_{0} \left( x \right) \mu_{0} \left( y \right) 
 = O \left( n^{2-d} \right),
\]
which, together with~\eqref{eq:xy_equal}, shows that $\mathcal{E}_{g} \left[ \mu_{0} \right] = O \left( n^{2-d} \right)$. 
This concludes the proof of~\eqref{eq:energy_estimate} for $d \geq 3$.

The case of $d = 2$ is similar, but the Green's function behaves differently, and we cannot neglect the energy of the mass distribution $\mu_{t}$ as we did for $d \geq 3$. 
We first bound $\mathcal{E}_{a} \left[ \mu_{t} \right]$ from above. 
Recall that $a \left( \root0 \right) = 0$ 
and that for every $x \neq \root0$ we have the estimate 
$a \left( x \right) \leq \tfrac{2}{\pi} \ln \left\| x \right\|_{2} + \kappa + C_{2} \left\| x \right\|_{2}^{-2}$ 
(see~\eqref{eq:green_estimates_d2}). 
We know that every $x$ in the support of $\mu_{t}$ satisfies $\left\| x \right\|_{1} \leq n$ and hence also $\left\| x \right\|_{2} \leq n$. 
Thus by the triangle inequality if both $x$ and $y$ are in the support of $\mu_{t}$ then $\left\| x - y \right\|_{2} \leq 2n$. 
Therefore 
\begin{align}
\mathcal{E}_{a} \left[ \mu_{t} \right] 
&= \sum_{x, y \in \mathbb{Z}^{d} : x \neq y} a \left( x - y \right) \mu_{t} \left( x \right) \mu_{t} \left( y \right) \notag \\
&\leq \left( \frac{2}{\pi} \ln \left( 2n \right) + \kappa + C_{2} \right) \sum_{x, y \in \mathbb{Z}^{d} : x \neq y} \mu_{t} \left( x \right) \mu_{t} \left( y \right) \notag \\
&= \frac{2}{\pi} \ln \left( n \right) + O \left( 1 \right). \label{eq:Ea_ub}
\end{align}
Next we bound from below the energy $\mathcal{E}_{a} \left[ \mu_{0} \right]$. Noting again that $a \left( \root0 \right) = 0$, we can write $\mathcal{E}_{a} \left[ \mu_{0} \right]$ similarly to~\eqref{eq:sum_shells}: 
\begin{equation}\label{eq:sum_shells_d2}
\mathcal{E}_{a} \left[ \mu_{0} \right] = \sum_{k=1}^{K} \sum_{ \left( x, y \right) \in E_{k} } a \left( x - y \right) \mu_{0} \left( x \right) \mu_{0} \left( y \right). 
\end{equation}
For $x \neq \root0$ we have the estimate 
$a \left( x \right) \geq \tfrac{2}{\pi} \ln \left\| x \right\|_{2} - C_{2} \left\| x \right\|_{2}^{-2}$ 
(see~\eqref{eq:green_estimates_d2} and note that $\kappa > 0$), 
and thus if $\left( x, y \right) \in E_{k}$ then 
\begin{equation}\label{eq:axy_lb}
a \left( x - y \right) \geq \frac{2}{\pi} \ln \left( 2 n \right) - \frac{2}{\pi} \ln \left( 2 \right) \times k - C_{2} \times \frac{2^{2k}}{4n^{2}}. 
\end{equation}
Plugging the estimate~\eqref{eq:axy_lb} into~\eqref{eq:sum_shells_d2} we get three terms which we can each estimate separately. 
First, observing that 
\[
\sum_{k=1}^{K} \sum_{ \left( x, y \right) \in E_{k} } \mu_{0} \left( x \right) \mu_{0} \left( y \right) 
= 1 - \sum_{x \in \mathbb{Z}^{d}} \mu_{0} \left( x \right)^{2} 
\geq 1 - \frac{4}{\mathrm{Vol} \left( B_{cn}^{(2)} \right)} 
= 1 - O \left( n^{-2} \right),
\]
we get that
\begin{equation}\label{eq:Ea_mu0_main}
\sum_{k=1}^{K} \sum_{ \left( x, y \right) \in E_{k} } \frac{2}{\pi} \ln \left( 2 n \right) \mu_{0} \left( x \right) \mu_{0} \left( y \right) 
= \frac{2}{\pi} \ln \left( n \right) - O \left( 1 \right). 
\end{equation}
For the second term in~\eqref{eq:axy_lb} we use~\eqref{eq:mass_of_shell} to obtain that 
\begin{equation}\label{eq:Ea_mu0_second}
\sum_{k=1}^{K} \sum_{ \left( x, y \right) \in E_{k} } \left( \frac{2}{\pi} \ln \left( 2 \right) \times k \right) \mu_{0} \left( x \right) \mu_{0} \left( y \right) 
= \sum_{k=1}^{K} O \left( k \times 2^{-2k} \right) 
= O \left( 1 \right). 
\end{equation}
For the third term in~\eqref{eq:axy_lb} we again use~\eqref{eq:mass_of_shell}, together with the fact that $K = \left\lceil \log_{2} \left( 2n \right) \right\rceil$, to get that 
\begin{equation}\label{eq:Ea_mu0_third}
\sum_{k=1}^{K} \sum_{ \left( x, y \right) \in E_{k} } \left( C_{2} \times \frac{2^{2k}}{4n^{2}} \right) \mu_{0} \left( x \right) \mu_{0} \left( y \right) 
= \frac{1}{n^{2}} \sum_{k=1}^{K} O \left( 1 \right) 
= O \left( \frac{\log \left( n \right) }{n^{2}} \right). 
\end{equation}
Putting together~\eqref{eq:Ea_mu0_main},~\eqref{eq:Ea_mu0_second}, and~\eqref{eq:Ea_mu0_third} with~\eqref{eq:sum_shells_d2} and~\eqref{eq:axy_lb} we obtain that 
\begin{equation}\label{eq:Ea_mu0_lb}
\mathcal{E}_{a} \left[ \mu_{0} \right] = \frac{2}{\pi} \ln \left( n \right) - O \left( 1 \right). 
\end{equation}
Finally, putting together~\eqref{eq:Ea_ub} and~\eqref{eq:Ea_mu0_lb} we obtain~\eqref{eq:energy_estimate} for $d=2$. 
\end{proof}

\section{A general upper bound} \label{sec:rw_ub} 

In this section we provide two proofs of Theorem~\ref{thm:volumeexittime}. 

\begin{proof}[Proof of Theorem~\ref{thm:volumeexittime} using random walk]
We write $\left| A \right| := \mathrm{Vol} \left( A \right)$ to abbreviate notation. 
Let $x_{1}, x_{2}, \dots, x_{\left| A \right|}$ denote the vertices of $A$ in some specific order. 
We define a sequence of toppling moves that proceeds in rounds by repeatedly cycling through the vertices of $A$ in this specified order and at each move toppling all of the mass that was at the given vertex at the beginning of the round. 
That is, letting $\mu_{0} := \delta_{\rootv}$, we let 
$\mu_{1} := T_{x_{1}}^{\mu_{0} \left( x_{1} \right)} \mu_{0}$, 
then 
$\mu_{2} := T_{x_{2}}^{\mu_{0} \left( x_{2} \right)} \mu_{1}$, 
and so on. 
In general, for a positive integer $i$, let $i^{*}$ be the unique integer in $\left\{ 1, 2, \dots, \left|A \right| \right\}$ such that $i - i^{*}$ is divisible by $\left| A \right|$. We then have that  
\begin{equation}\label{eq:seq_topplings_rw}
\mu_{i} := T_{x_{i^{*}}}^{m_{i}} \mu_{i-1}, \qquad  \text{with } \quad m_{i} = \mu_{i-i^{*}} \left( x_{i^{*}} \right).
\end{equation}
We call each group of $\left|A\right|$ toppling moves a \emph{round} of the toppling process. 

Let $\left\{ Z_{t} \right\}_{t \geq 0}$ denote the random walk on $G$ that is killed when it exits $A$, i.e., $Z_{t} = X_{t \wedge T_{A}}$, with initial condition $Z_{0} = \rootv$.  
Observe that all the toppling moves of a given round can be executed in parallel, since the mass that is toppled at each vertex only depends on the mass distribution at the beginning of the round. 
Since all of the mass that is present in $A$ at the beginning of the round is toppled, each round of the toppling process defined in~\eqref{eq:seq_topplings_rw} perfectly simulates a step of the killed random walk $\left\{ Z_{t} \right\}_{t \geq 0}$. 
That is, for every nonnegative integer $t$, the measure $\mu_{t \left|A\right|}$ agrees with the distribution of $Z_{t}$. 

Let 
\[
M := \inf \left\{ i \geq 0 : \mu_{i\left|A\right|} \left( A \right) \leq 1 - p \right\}
\]
denote the first time that the distribution of the killed random walk has mass at least $p$ outside of the set $A$. 
By the definition of the exit time $T_{A}$ we have that 
\begin{equation}\label{eq:expectation_TA}
\mathbb{E}_{\rootv} \left[ T_{A} \right] 
= \sum_{k=1}^{\infty} \mathbb{P}_{\rootv} \left( T_{A} \geq k \right) 
= \sum_{k=1}^{\infty} \mathbb{P}_{\rootv} \left( Z_{k-1} \in A \right) 
= \sum_{k=1}^{\infty} \mu_{\left( k - 1 \right) \left| A \right|} \left( A \right).
\end{equation}
Now by the definition of $M$ we have that for every $m < M$, 
the measure $\mu_{m\left|A\right|}$ satisfies 
$\mu_{m\left|A\right|} \left( A \right) > 1 - p$. 
Therefore keeping only the first $M$ terms in the sum in~\eqref{eq:expectation_TA} we obtain the bound  
\[
\mathbb{E}_{\rootv} \left[ T_{A} \right] 
\geq \sum_{k=1}^{M} \mu_{\left( k - 1 \right) \left| A \right|} \left( A \right) 
> M \left( 1 - p \right).
\]
By the definition of $M$ this immediately implies that 
\[
N_{p} \left( G, A, \delta_{\rootv} \right) \leq M \times \left| A \right| < \left( 1 - p \right)^{-1} \mathbb{E}_{\rootv} \left[ T_{A} \right] \times \left| A \right|. \qedhere 
\]
\end{proof}

Theorem~\ref{thm:volumeexittime} can also be proven using a greedy algorithm, similarly to the proof of the greedy upper bound on $\mathbb{Z}^{d}$ presented in Section~\ref{sec:Zd_ub}. The only part of that proof that was specific to $\mathbb{Z}^{d}$ was the use of the second moment of the mass distribution. In particular, the key property of the second moment that we used was that $\Delta \left\| x \right\|_{2}^{2} = 1$ for every $x \in \mathbb{Z}^{d}$. For a general graph $G = (V, E)$ and a subset of the vertices $A \subset V$, the expected first exit time from $A$ starting from a given vertex is a function whose discrete Laplacian is constant on $A$. This is because by conditioning on the first step of the random walk we have that 
\begin{equation}\label{eq:exittime_onestep}
\mathbb{E}_{x} \left[ T_{A} \right] = 1 + \frac{1}{\left| \mathcal{N}_{x} \right|} \sum_{y \sim x} \mathbb{E}_{y} \left[ T_{A} \right]
\end{equation}
for every $x \in A$. 

\begin{proof}[Proof of Theorem~\ref{thm:volumeexittime} using a greedy algorithm]
Consider the following greedy algorithm for choosing toppling moves:  
until the mass outside of $A$ is at least $p$, 
choose $v \in A$ with the largest mass in $A$ (break ties arbitrarily) 
and topple the full mass at $v$. 
Let $\mu_{0} \equiv \delta_{\rootv}, \mu_{1}, \mu_{2}, \dots$ denote the resulting mass distributions, 
let $v_{i}$ denote the vertex that was toppled to get from $\mu_{i-1}$ to $\mu_{i}$, 
and let $m_{i}$ denote the mass that was toppled at this step. 
By~\eqref{eq:toppling_move} we can then write 
$
\mu_{i} = \mu_{i-1} + m_{i} \Delta \delta_{v_{i}}. 
$
Furthermore, let $t$ denote the number of moves necessary for this greedy algorithm to transport mass $p$ to distance at least $n$ from the origin, 
i.e., $t = \min \left\{ i \geq 1 : \mu_{i} \left( A \right) \leq 1 - p \right\}$. 

For $x \in V$, let 
$h \left( x \right) := - \mathbb{E}_{x} \left[ T_{A} \right]$. 
We have $h \left( x \right) = 0$ for every $x \notin A$, 
and, by~\eqref{eq:exittime_onestep}, we have that $\Delta h \left( x \right) = 1$ for every $x \in A$. 
For a mass distribution $\mu$ define 
$\widetilde{M} \left[ \mu \right] := \sum_{x \in V} \mu \left( x \right) h \left( x \right)$. 
We have that 
$\widetilde{M} \left[ \mu_{0} \right] = - \mathbb{E}_{\rootv} \left[ T_{A} \right]$ 
and 
$\widetilde{M} \left[ \mu_{t} \right] \leq 0$. 
We first compute how $\widetilde{M}$ changes after each toppling move:
\[
\widetilde{M} \left[ \mu_{i} \right] - \widetilde{M} \left[ \mu_{i-1} \right] 
= \sum_{x \in V} \mu_{i} \left( x \right) h \left( x \right) - \sum_{x \in V} \mu_{i-1} \left( x \right) h \left( x \right) 
= m_{i} \sum_{x \in V} \Delta \delta_{v_{i}} \left( x \right) \cdot h \left( x \right). 
\]
Now by definition we have that 
\begin{align*}
 \sum_{x \in V} \Delta \delta_{v_{i}} \left( x \right) \cdot h \left( x \right) 
 &= \sum_{x \in V} \left( - \delta_{v_{i}} \left( x \right) + \frac{1}{\left| \mathcal{N}_{v_{i}} \right|} \sum_{y \sim v_{i}} \delta_{y} \left( x \right) \right) h \left( x \right) \\
 &= - h \left( v_{i} \right) + \frac{1}{\left| \mathcal{N}_{v_{i}} \right|} \sum_{y \sim v_{i}} h \left( y \right) = \Delta h \left( v_{i} \right) = 1,
\end{align*}
where the last equality follows from the fact that $\Delta h \left( x \right) = 1$ for every $x \in A$. 
Putting the previous two displays together we obtain that 
\[
\widetilde{M} \left[ \mu_{i} \right] - \widetilde{M} \left[ \mu_{i-1} \right] = m_{i}
\]
for every $i \leq t$. 
Note that the greedy choice implies that $m_{i} > \left( 1- p \right) / \left| A \right|$ for every $i \leq t$. 
Therefore we obtain that 
\[
\mathbb{E}_{\rootv} \left[ T_{A} \right] 
\geq \widetilde{M} \left[ \mu_{t} \right] - \widetilde{M} \left[ \mu_{0} \right] 
= \sum_{i=1}^{t} \left( \widetilde{M} \left[ \mu_{i} \right] - \widetilde{M} \left[ \mu_{i-1} \right] \right) 
= \sum_{i=1}^{t} m_{i} 
> t \times \frac{1-p}{\left|A\right|},
\]
and the result follows by rearranging this inequality. 
\end{proof}

\section{Controlled diffusion on the comb} \label{sec:proofs_comb} 

The general upper bound given by Theorem~\ref{thm:volumeexittime} applied directly to the comb $\mathbb{C}_{2}$ gives a bound of 
\begin{equation}\label{eq:comb_general_ub}
 N_{p} \left( \mathbb{C}_{2}, B_{n}, \delta_{\root0} \right) \leq C \left( 1 - p \right)^{-1} n^{4} 
\end{equation}
for some constant $C$, since $\mathrm{Vol} \left( B_{n} \right) = \Theta \left( n^{2} \right)$ and $\mathbb{E}_{\root0} \left[ T_{B_{n}} \right] = \Theta \left( n^{2} \right)$. 
However, this bound is not tight. 
Recall that the general upper bound that gives~\eqref{eq:comb_general_ub} is proven by simulating a random walk within $B_{n}$. 
The key observation that improves~\eqref{eq:comb_general_ub} to a tight bound is that one can restrict the random walk on $\mathbb{C}_{2}$ to the rectangle 
$R_{C,n} := \left[ - C \sqrt{n}, C \sqrt{n} \right] \times \left[ - n, n \right]$ for large enough $C$. 
This is because with probability close to $1$, the random walk will exit $B_{n}$ before it exits the rectangle $R_{C,n}$. 
Since $\mathrm{Vol} \left( R_{C,n} \right) = \Theta \left( n^{3/2} \right)$, this gives the improved upper bound of $O \left( n^{7/2} \right)$. 

To obtain a matching lower bound, we first smooth the mass distribution by simulating the random walk killed when it exits the rectangle 
$\left[ - C \sqrt{n}, C \sqrt{n} \right] \times \left( - n/2, n/2 \right)$. 
The resulting mass distribution $\mu$ has almost all of its mass on the ``ends of the teeth'', i.e., on the set 
\[
S := \left\{ \left( i, \pm \frac{n}{2} \right) : \left| i \right| \leq C \sqrt{n} \right\}.
\]
Moreover, most of the mass is roughly uniformly spread on $S$, 
in the sense that $\mu \left( x \right) = O \left( 1 / \sqrt{n} \right)$ for every $x \in S$ (after potentially throwing away a tiny constant mass). 
So in order to move a constant mass $p$ to distance $n$ from the origin $\root0$, 
we need to move a constant fraction of the mass present at $\Omega \left( \sqrt{n} \right)$ points in $S$. 
Since each ``tooth'' of the comb is locally a line, this requires $\Omega \left( n^{3} \right)$ toppling moves along each tooth (by Theorem~\ref{thm:zd} for $d = 1$, proven in~\cite{overhang}), 
resulting in $\Omega \left( n^{7/2} \right)$ toppling moves in total.

The rest of this section makes the two preceding paragraphs precise and proves Theorem~\ref{thm:comb}. 
Let $\left\{ X_{t} \right\}_{t \geq 0}$ denote random walk on $\mathbb{C}_{2}$ started at the origin, i.e., with $X_{0} = \root0$. 
We write $R \equiv R_{C,n}$ when the implied parameters are clear from the context. 
Let $T_{R}$ denote the first exit time of the random walk $\left\{ X_{t} \right\}_{t \geq 0}$ from $R$ 
and let $Z_{t} := X_{t \wedge T_{R}}$ denote the random walk killed when it exits $R$. 
We write $X_{t} = \left( X_{t}^{(1)}, X_{t}^{(2)} \right)$ and $Z_{t} = \left( Z_{t}^{(1)}, Z_{t}^{(2)} \right)$ for the coordinates of $X_{t}$ and $Z_{t}$. 

The following lemma says that by making $C$ large enough, one can make the probability that the random walk exits $R_{C,n}$ along one of the ``teeth'' arbitrarily close to $1$. 
\begin{lemma}\label{lem:comb_rectangle_teeth}
For every $\eps > 0$ there exists $C = C \left( \eps \right) < \infty$ such that 
\[
\p_{\root0} \left( X_{T_{R_{C,n}}}^{(2)} = 0 \right) \leq \eps. 
\]
\end{lemma}
\begin{proof}
It suffices to show that if we run the random walk on $\mathbb{C}_{2}$ for $c n^{2}$ steps, where $c$ is large enough, 
then the probability that the random walk has not yet reached $\pm n$ in the second coordinate is small 
and the probability that the random walk has reached $\pm C \sqrt{n}$ in the first coordinate is also small. 
More precisely, the statement follows from the following two inequalities:
\begin{align}
\p_{\root0} \left( \left| X_{t}^{(2)} \right| \leq n \text{ for every } t \leq c n^{2} \right) &\leq \eps / 2, \label{eq:not_exit_teeth} \\
\p_{\root0} \left( \left| X_{t}^{(1)} \right| \geq C \sqrt{n} \text{ for some } t \leq c n^{2} \right) &\leq \eps / 2. \label{eq:exit_side}
\end{align}
Note that $\left\{ X_{t}^{(2)} \right\}_{t \geq 0}$ is Markovian: when away from $0$ it behaves like simple symmetric random walk on $\mathbb{Z}$ and at $0$ it becomes lazy, i.e., it stays put with probability $1/2$, and it jumps to $\pm 1$ with probability $1/4$ each. Therefore~\eqref{eq:not_exit_teeth} follows from classical random walk estimates (for instance, it follows from the central limit theorem, see, e.g.,~\cite[Theorem~2.9]{revesz2005random}), provided $c = c \left( \eps \right)$ is large enough.  

Now fix $c$ such that~\eqref{eq:not_exit_teeth} holds. Again by classical estimates (see, e.g.,~\cite[Theorem~9.11]{revesz2005random}) there exists a constant $c'$ such that 
\[
\# \left\{ t : t \leq c n^{2}, X_{t}^{(2)} = 0 \right\} \leq c' n 
\]
with probability at least $1 - \eps / 4$. 
Note that $\left\{ X_{t}^{(1)} \right\}_{t \geq 0}$ only moves at times when $X_{t}^{(2)} = 0$, and when it does, it moves according a lazy random walk, staying in put with probability $1/2$. 
Let $\left\{ Y_{t} \right\}_{t \geq 0}$ denote such a lazy random walk. 
By classical estimates we have that 
\[
\p_{\root0} \left( \left| Y_{t} \right| \geq C \sqrt{n} \text{ for some } t \leq c' n \right) \leq \eps / 4
\]
provided that $C$ is large enough. Putting everything together gives us~\eqref{eq:exit_side}. 
\end{proof}

With this lemma in hand we are now ready to prove 
Theorem~\ref{thm:comb}. 
We start with the upper bound and we again give two proofs, one using random walk and one using a greedy algorithm. 

\begin{proof}[Proof of the upper bound of Theorem~\ref{thm:comb} using random walk]
Fix $\eps \in \left( 0,  1 - p \right)$, let $C = C \left( \eps \right)$ be the constant given by Lemma~\ref{lem:comb_rectangle_teeth}, and let $R := R_{C,n}$. 
Just like in the proof of Theorem~\ref{thm:volumeexittime}, we define a sequence of toppling moves $\mu_{0} := \delta_{\root0}, \mu_{1}, \mu_{2}, \dots$ that simulate the killed random walk $\left\{ Z_{t} \right\}_{t \geq 0}$, 
i.e., for every nonnegative integer $t$, the distribution $\mu_{t \left| R \right|}$ agrees with the distribution of~$Z_{t}$. 

Let 
\[
M := \inf \left\{ i \geq 0 : \mu_{i \left| R \right|} \left( R \right) \leq 1 - p - \eps  \right\}
\]
denote the first time that the distribution of the killed random walk has mass at least $p + \eps$ outside of the rectangle $R$. 
By Lemma~\ref{lem:comb_rectangle_teeth} we have that 
\[
\mu_{M \left|R\right|} \left( \left\{ \left( - C \sqrt{n} - 1, 0 \right) \right\} \cup \left\{ \left( C \sqrt{n} + 1, 0 \right) \right\} \right) 
\leq \p_{\root0} \left( X_{T_{R}}^{\left( 2 \right)} = 0 \right) \leq \eps,
\]
i.e., there is mass at most $\eps$ that is not at the ``ends of the teeth'' of $R$. Since every other vertex in the support of $\mu_{M \left| R \right|}$ that is outside of $R$ has distance at least $n$ from the origin, it follows that $\mu_{M \left| R \right|} \left( B_{n} \right) \leq 1 - p$, which implies that 
\[
 N_{p} \left( \mathbb{C}_{2}, B_{n}, \delta_{\root0} \right) \leq M \left| R \right|.
\]
Just like in the proof of Theorem~\ref{thm:volumeexittime}, one can show that 
\[
 M < \left( 1 - p - \eps \right)^{-1} \mathbb{E}_{\root0} \left[ T_{R} \right].
\]
The upper bound now follows by putting together the previous two displays and using the facts that 
$\left| R \right| = \Theta \left( n^{3/2} \right)$ 
and 
$\mathbb{E}_{\root0} \left[ T_{R} \right] = \Theta \left( n^{2} \right)$. 
\end{proof}

\begin{proof}[Proof of the upper bound of Theorem~\ref{thm:comb} using a greedy algorithm]
Fix $\eps \in \left( 0,  1 - p \right)$, let $C = C \left( \eps \right)$ be the constant given by Lemma~\ref{lem:comb_rectangle_teeth}, and let $R := R_{C,n}$. 
Consider the following greedy algorithm for choosing toppling moves: 
until the mass outside of $R$ is at least $p + \eps$, 
choose $v \in R$ with the largest mass in $R$ (break ties arbitrarily) 
and topple the full mass at $v$. 
Let 
$\mu_{0} \equiv \delta_{\root0}, \mu_{1}, \mu_{2}, \dots$ 
denote the resulting mass distributions, 
let $v_{i}$ denote the vertex that was toppled to get from $\mu_{i-1}$ to $\mu_{i}$, 
and let $m_{i}$ denote the mass that was toppled at this step. 
Furthermore, let $t$ denote the number of moves necessary for this greedy algorithm to transport mass $p + \eps$ outside of $R$, 
i.e., 
$t = \min \left\{ i \geq 1 : \mu_{i} \left( R \right) \leq 1 - p - \eps \right\}$. 

Just as in the proof of Theorem~\ref{thm:greedy_ub_Zd} we can compute how the second moment of the mass distribution changes after each toppling move and we obtain that 
\[
 M_{2} \left[ \mu_{i} \right] - M_{2} \left[ \mu_{i-1} \right] = m_{i}. 
\]
The greedy choice implies that for every $i \leq t$ we must have that 
\[
 m_{i} \geq \frac{\mu_{i-1} \left( R \right)}{\left| R \right|} > \frac{1-p-\eps}{\left| R \right|}. 
\]
This gives us the following lower bound on the second moment of $\mu_{t}$: 
\[
M_{2} \left[ \mu_{t} \right] 
= \sum_{i=1}^{t} \left( M_{2} \left[ \mu_{i} \right] - M_{2} \left[ \mu_{i-1} \right] \right) 
= \sum_{i=1}^{t} m_{i} 
> t \times \frac{1-p-\eps}{\left| R \right|}.
\]
On the other hand, 
there exists a constant $C' < \infty$ such that 
$\left\| v \right\|_{2}^{2} \leq C' n^{2}$ for every $v \in \mathbb{C}_{2}$ such that $\mu_{t} \left( v \right) > 0$, 
which implies that 
$M_{2} \left[ \mu_{t} \right] \leq C' n^{2}$. 
Combining this with the display above 
we obtain that 
$t < C' n^{2} \times \left| R \right| / \left( 1 - p - \eps \right)$. 
Since $\left| R \right| = \Theta \left( n^{3/2} \right)$ 
we thus have that 
$t = O \left( n^{7/2} \right)$. 

What remains to show is that the mass distribution $\mu_{t}$ has mass at least $p$ at distance at least $n$ from the origin, i.e., that $\mu_{t} \left( B_{n} \right) \leq 1 - p$. 
Note that there are only two vertices in the vertex boundary of $R$ that are at distance less than $n$ from the origin: 
$\left( -C \sqrt{n} -1, 0 \right)$ and $\left( C \sqrt{n} + 1, 0 \right)$. 
Thus we have that 
\[
 \mu_{t} \left( B_{n} \right) 
 \leq \mu_{t} \left( R \right) + \mu_{t} \left( \left( -C \sqrt{n} -1, 0 \right) \right) + \mu_{t} \left( \left( C \sqrt{n} + 1, 0 \right) \right), 
\]
and since $\mu_{t} \left( R \right) \leq 1 - p - \eps$, what remains to show is that 
\begin{equation}\label{eq:mass_side_eps}
\mu_{t} \left( \left( -C \sqrt{n} -1, 0 \right) \right) + \mu_{t} \left( \left( C \sqrt{n} + 1, 0 \right) \right) \leq \eps. 
\end{equation}
For $x \in \mathbb{C}_{2}$ let 
$h \left( x \right) := \p_{x} \left( X_{T_{R}}^{(2)} = 0 \right)$. 
By Lemma~\ref{lem:comb_rectangle_teeth} we have that 
$h \left( \root0 \right) \leq \eps$, 
and hence 
$\sum_{x \in \mathbb{C}_{2}} h \left( x \right) \mu_{0} \left( x \right) \leq \eps$. 
The function $h$ is harmonic in $R$, which implies that 
$\sum_{x \in \mathbb{C}_{2}} h \left( x \right) \mu_{i} \left( x \right) 
= \sum_{x \in \mathbb{C}_{2}} h \left( x \right) \mu_{i-1} \left( x \right)$ 
for every $i \geq 1$, 
and hence 
$\sum_{x \in \mathbb{C}_{2}} h \left( x \right) \mu_{t} \left( x \right) \leq \eps$. 
The inequality~\ref{eq:mass_side_eps} then immediately follows from the fact that 
$h \left( \left( -C \sqrt{n} -1, 0 \right) \right) = h \left( \left( C \sqrt{n} + 1, 0 \right) \right) = 1$. 
\end{proof}

\begin{proof}[Proof of the lower bound of Theorem~\ref{thm:comb}]
Given $p \in \left( 0, 1 \right)$, let $\eps := p / 4$. 
In the following we fix $c = c \left( \eps \right)$ and $C = C \left( \eps \right)$ to be large enough constants; 
we shall see soon the specific criterion for choosing these constants. 

We start by smoothing the initial mass distribution appropriately. Define the rectangle 
$R' \equiv R'_{C,n} := \left[ - C \sqrt{n}, C \sqrt{n} \right] \times \left( - n / 2 , n / 2 \right)$ 
and let 
$Z'_{t} := X_{t \wedge T_{R'}}$ denote the random walk killed when it exits $R'$. 
Starting with the initial mass distribution $\delta_{\root0}$, we apply a sequence of 
$c n^{2} \times \mathrm{Vol} \left( R' \right)$ 
toppling moves that simulate $c n^{2}$ steps of the killed random walk $\left\{ Z'_{t} \right\}_{t \geq 0}$, 
to arrive at a new mass distribution $\mu$. 
In the same way as in the proof of Lemma~\ref{lem:comb_rectangle_teeth}, we can argue that most of the mass of the resulting measure $\mu$ is on the ``ends of the teeth'', i.e., it is on the set 
\[
S := \left\{ \left( i, \pm \frac{n}{2} \right) : \left| i \right| \leq C \sqrt{n} \right\}.
\] 
More precisely, if $c$ and $C$ are chosen appropriately, then $\mu \left( S \right) \geq 1 - \eps$. 
Furthermore, most of the mass is roughly uniformly spread on $S$.  
Specifically, we claim that there exists a constant $K$ such that 
we can write the mass measure $\mu$ restricted to $S$ as the sum of two mass measures, 
$\mu|_{S} = \mu_{1} + \mu_{2}$,  
such that 
\begin{equation}\label{eq:comb_spread}
\mu_{1} \left( x \right) \leq \frac{K}{\sqrt{n}}, \quad \forall x \in S, \qquad \textrm{and } \qquad \mu_{2} \left( S \right) \leq \eps.
\end{equation}

Before proving~\eqref{eq:comb_spread}, we show how to conclude the proof assuming that~\eqref{eq:comb_spread} holds. 
First of all, from Corollary~\ref{cor:smoothing} we have that 
$N_{p} \left( \mathbb{C}_{2}, B_{n}, \delta_{\root0} \right) \geq N_{p} \left( \mathbb{C}_{2}, B_{n}, \mu \right)$, 
so it suffices to bound from below this latter quantity. 
Now suppose that a sequence of toppling moves takes $\mu$ to a mass distribution $\mu'$ satisfying $\mu' \left( B_{n} \right) \leq 1 - p$, 
and for $x \in S$ let 
$\nu \left( x \right) \in \left[ 0, \mu \left( x \right) \right]$ 
denote the amount of mass that was originally (under $\mu$) at $x$, but through the toppling moves was transported outside of $B_{n}$. 
We can write 
$\nu \left( x \right) = \nu_{1} \left( x \right)  + \nu_{2} \left( x \right)$ 
in accordance with how we have 
$\mu \left( x \right) = \mu_{1} \left( x \right)  + \mu_{2} \left( x \right)$. 
Since $\mu \left( S \right) \geq 1 - \eps$ and $\mu' \left( B_{n} \right) \leq 1 - p$, we must have that 
\begin{equation}\label{eq:sum_nu}
\sum_{x \in S} \nu \left( x \right) \geq p - \eps. 
\end{equation}
Since 
$\nu_{2} \left( S \right) \leq \mu_{2} \left( S \right) \leq \eps$, 
we must then have that 
\begin{equation}\label{eq:sum_nu1}
\sum_{x \in S} \nu_{1} \left( x \right) \geq p - 2\eps. 
\end{equation}
Let $S_{\mathrm{lg}} := \left\{ x \in S \, : \, \nu_{1} \left( x \right) \geq \eps / ( 5 C \sqrt{n} ) \right\}$ 
and $S_{\mathrm{sm}} := S \setminus S_{\mathrm{lg}}$, 
and break the sum in~\eqref{eq:sum_nu1} into two parts accordingly. 
Using that 
$\left| S \right| = 4C \sqrt{n} + 2 \leq 5C \sqrt{n}$, 
we have that 
$\sum_{x \in S_{\mathrm{sm}}} \nu_{1} \left( x \right) \leq \eps$, 
and so 
\[
 \sum_{x \in S_{\mathrm{lg}}} \nu_{1} \left( x \right) \geq p - 3 \eps = p/4. 
\]
On the other hand,~\eqref{eq:comb_spread} implies that 
\[
 \sum_{x \in S_{\mathrm{lg}}} \nu_{1} \left( x \right) \leq \left| S_{\mathrm{lg}} \right| \times \frac{K}{\sqrt{n}} 
\]
and so we must have that 
$\left| S_{\mathrm{lg}} \right| \geq \frac{p}{4K} \sqrt{n}$. 
Notice that for every $x \in S_{\mathrm{lg}}$ we have that 
$\nu_{1} \left( x \right) / \mu_{1} \left( x \right) \geq \eps / (5CK)$, 
i.e., a constant fraction of the mass at $x$ (under $\mu_{1}$) is transported outside of $B_{n}$. 
In order to transport mass from $x = \left( x_{1}, x_{2} \right) \in S$ to outside of $B_{n}$, 
the mass necessarily has to go through 
either $\left( x_{1}, x_{2} + n/4 \right)$ 
or $\left( x_{1}, x_{2} - n/4 \right)$. 
Since the graph between these two points is a line of length $\Omega \left( n \right)$, 
we know from Theorem~\ref{thm:zd} for $d = 1$ (proven in~\cite{overhang}) that
$\Omega \left( n^{3} \right)$ 
toppling moves are necessary to do this. 
Since this holds for every $x \in S_{\mathrm{lg}}$, 
we see that $\Omega \left( n^{7/2} \right)$ toppling moves are necessary altogether.

Finally, we turn back to proving~\eqref{eq:comb_spread}. 
First, note that there exists $\delta = \delta \left( \eps \right)$ such that 
with probability at least $1 - \eps / 2$, 
the killed random walk $\left\{ Z_{t}' \right\}_{t \geq 0}$ has not exited the rectangle $R'$ by time $\delta n^{2}$ (this follows by classical estimates for simple random walk, see, e.g.,~\cite{revesz2005random}). 
On this event, which we shall denote by $\mathcal{A}$, the killed random walk $\left\{ Z_{t}' \right\}_{t = 0}^{\delta n^{2}}$ and the simple random walk $\left\{ X_{t} \right\}_{t=0}^{\delta n^{2}}$ agree. 
Now let $N_{t}^{Z'}$ denote the number of visits to the $x$ axis of the killed random walk until time $t$, i.e., 
\[
 N_{t}^{Z'} := \# \left\{ k \in \left\{ 0, 1, \dots, t \right\} : {Z'}_{k}^{(2)} = 0 \right\},
\]
and similarly define $N_{t}^{X}$ for the simple random walk. 
Under the event $\mathcal{A}$, we have that 
\[
 N_{\delta n^{2}}^{Z'} = N_{\delta n^{2}}^{X}.
\]
By classical estimates on the local time at $0$ (see, e.g.,~\cite[Theorem~9.11]{revesz2005random}), 
there exists $\gamma = \gamma \left( \eps \right)$ 
such that with probability at least $1 - \eps / 2$, 
we have that 
\begin{equation}\label{eq:local_time_lb}
N_{\delta n^{2}}^{X} \geq \gamma n. 
\end{equation}
Denote by $\mathcal{B}$ the event that the inequality in~\eqref{eq:local_time_lb} holds 
and note that $\p \left( \mathcal{A} \cap \mathcal{B} \right) \geq 1 - \eps$. 
In the following we assume that the event $\mathcal{A} \cap \mathcal{B}$ holds; 
whatever happens on the event $\left( \mathcal{A} \cap \mathcal{B} \right)^{c}$ we put into the mass measure $\mu_{2}$, which hence has mass at most $\eps$.

Under the event $\mathcal{A} \cap \mathcal{B}$ we have that 
$N := N_{c n^{2}}^{Z'} \geq N_{\delta n^{2}}^{Z'} = N_{\delta n^{2}}^{X} \geq \gamma n$. 
Let $\left\{ Y_{t} \right\}_{t \geq 0}$ denote a lazy random walk on $\mathbb{Z}$ that stays put with probability $1/2$, and otherwise does a step according to simple random walk, just like in the proof of Lemma~\ref{lem:comb_rectangle_teeth}. 
Conditioned on $N$, we have that ${Z'}_{c n^{2}}^{(1)}$ has the same distribution as $Y_{N}$. For fixed $N$, the local limit theorem says that there exists $K'$ such that 
\[
\sup_{\ell \in \mathbb{Z}} \p \left( Y_{N} = \ell \right) \leq \frac{K'}{\sqrt{N}}. 
\]
Hence there exists $K$ such that 
\[
\sup_{\ell \in \mathbb{Z}} \p \left( {Z'}_{c n^{2}}^{(1)} = \ell \, \middle| \, \mathcal{A} \cap \mathcal{B} \right) \leq \frac{K}{\sqrt{n}},
\]
which implies the claim. 
\end{proof}

\section{Graphs where random walk has positive speed} \label{sec:proofs_pos_speed-entropy} 

In this section we study graphs on which simple random walk has positive speed. 
As a warm-up, we study $d$-ary trees in Section~\ref{sec:proofs_dary_trees}, followed by general results in Section~\ref{sec:proofs_pos_speed_general}. 
We then apply the general results to two examples: Galton-Watson trees (Section~\ref{sec:proofs_gw_trees}) and product of trees (Section~\ref{sec:proofs_product_trees}). 
The main observation for these latter results is that 
in these cases one can a priori specify an exponentially small subset of the vertices of the ball of radius $n$ 
with the property that the random walk on the graph started from the center of the ball does not exit this subset with probability close to $1$. 
Thus simple random walk can be simulated approximately by performing toppling moves only on this exponentially small subset of $B_{n}$, leading to much better bounds than the general upper bound of Theorem~\ref{thm:volumeexittime}.

\subsection{$d$-ary trees} \label{sec:proofs_dary_trees} 

The general upper bound of Theorem~\ref{thm:volumeexittime} applied directly to the $d$-ary tree $\mathbb{T}_{d}$ gives 
\[
 N_{p} \left( \mathbb{T}_{d}, B_{n}, \delta_{\rootrho} \right) < C \left( 1 - p \right)^{-1} \cdot n \cdot d^{n}
\]
for some constant $C < \infty$, since $\mathrm{Vol} \left( B_{n} \right) = \Theta \left( d^{n} \right)$ and $\mathbb{E}_{\rootrho} \left[ T_{B_{n}} \right] = \Theta \left( n \right)$. 
However, this bound is not tight, as Theorem~\ref{thm:d-ary} states that 
$N_{p} \left( \mathbb{T}_{d}, B_{n}, \delta_{\rootrho} \right) = \Theta \left( d^{n} \right)$. 
This example is interesting because the factor coming from the exit time of the random walk is completely absent from $N_{p} \left( \mathbb{T}_{d}, B_{n}, \delta_{\rootrho} \right)$. 
The proof requires a more careful analysis of the greedy algorithm. 

In the rest of this subsection we prove Theorem~\ref{thm:d-ary}, starting with the lower bound. 
We define the level of a vertex $v \in \mathbb{T}_{d}$ to be its distance from the root: 
$\ell \left( v \right) := d_{\mathbb{T}_{d}} \left( v, \rootrho \right)$. 

\begin{proof}[Proof of the lower bound in~Theorem~\ref{thm:d-ary}]
We begin by smoothing the initial mass distribution in such a way that most of the mass is on the vertices at level $n-1$, where it is uniformly spread. 
More precisely, for any $\eps > 0$ it is possible to obtain, via a finite sequence of toppling moves, a mass distribution $\mu$ such that 
$\mu \left( v \right) \in \left( \left( 1 - \eps \right) d^{-(n-1)}, d^{-(n-1)} \right)$ 
for every vertex $v$ at level $n - 1$. 
By Corollary~\ref{cor:smoothing} we have that $N_{p} \left( \mathbb{T}_{d}, B_{n}, \delta_{\rootrho} \right) \geq N_{p} \left( \mathbb{T}_{d}, B_{n}, \mu \right)$, so it suffices to bound from below this latter quantity. 

Fix $\eps \in \left( 0, p \right)$. In order to transport mass at least $p$ to level $n$ starting from $\mu$, it is necessary to transport mass at least $p - \eps$ to level $n$ from the vertices at level $n-1$. 
However, each vertex at level $n-1$ has mass at most $d^{-(n-1)}$. Hence mass from at least $(p - \eps) d^{n-1}$ vertices at level $n-1$ needs to transported to level $n$, and this requires at least $(p - \eps) d^{n-1}$ toppling moves. 
Hence $N_{p} \left( \mathbb{T}_{d}, B_{n}, \mu \right) \geq (p-\eps) d^{n-1}$. 
\end{proof}

The greedy algorithm provides an upper bound of the correct order. 
In order to analyze it we study the average level of a mass distribution $\mu$, defined as 
\[
M_{1} \left[ \mu \right] := \sum_{v \in \mathbb{T}_{d}} \mu \left( v \right) \ell \left( v \right). 
\]
We will make use of the following lemma, which states that if the average level is not too large, then there must be a reasonably large mass at some vertex. 
\begin{lemma}\label{lem:existence_of_large_mass}
If $\mu$ is a mass distribution on $\mathbb{T}_{d}$ such that 
$M_{1} \left[ \mu \right] \leq \ell$, 
then there exists $v \in \mathbb{T}_{d}$ such that 
$\ell \left( v \right) \leq \ell$ and 
$\mu \left( v \right) \geq d^{-(\ell +1)} / 4$. 
\end{lemma}
\begin{proof}
We prove the statement by contradiction. 
Suppose that $\mu \left( v \right) < d^{-(\ell + 1)} / 4$ for every $v \in \mathbb{T}_{d}$ such that $\ell \left( v \right) \leq \ell$; our goal is to show that then $M_{1} \left[ \mu \right] > \ell$. 
To bound $M_{1} \left[ \mu \right]$ from below, we can first bound $\ell \left( v \right)$ by $\ell + 1$ for every $v$ such that $\ell \left( v \right) \geq \ell + 1$ to obtain that 
\begin{align*}
M_{1} \left[ \mu \right] 
&\geq \sum_{v : \ell \left( v \right) \leq \ell} \mu \left( v \right) \ell \left( v \right) + \left( \ell + 1 \right) \left( 1 - \sum_{v : \ell \left( v \right) \leq \ell} \mu \left( v \right) \right) \\
&= \ell + 1 - \sum_{v : \ell \left( v \right) \leq \ell} \mu \left( v \right) \left( \ell + 1 - \ell \left( v \right) \right). 
\end{align*}
Using the assumption that $\mu \left( v \right) < d^{-(\ell + 1)} / 4$ for every $v \in \mathbb{T}_{d}$ such that $\ell \left( v \right) \leq \ell$, we thus have that 
\[
M_{1} \left[ \mu \right] \geq \ell + 1 - \frac{1}{4} d^{- (\ell + 1)} \sum_{v : \ell \left( v \right) \leq \ell} \left( \ell + 1 - \ell \left( v \right) \right). 
\]
Finally, we have that 
\[
\sum_{v : \ell \left( v \right) \leq \ell} \left( \ell + 1 - \ell \left( v \right) \right) = \sum_{k = 0}^{\ell} \left( \ell + 1 - k \right) d^{k} 
= \frac{1}{d-1} \left[ d \cdot \frac{d^{\ell+1} - 1}{d-1} - \left( \ell + 1 \right) \right] \leq 2 d^{\ell + 1},
\]
and so $M_{1} \left[ \mu \right] \geq \ell + 1/2$. 
\end{proof}

\begin{proof}[Proof of the upper bound in~Theorem~\ref{thm:d-ary}]
Consider the following greedy algorithm for choosing toppling moves: until the mass outside of $B_{n}$ is at least $p$, choose $v \in B_{n}$ with the largest mass in $B_{n}$ (break ties arbitrarily) and topple the full mass at $v$. 
Let $\mu_{0} \equiv \delta_{\rootrho}, \mu_{1}, \mu_{2}, \dots$ denote the resulting mass distributions, 
let $v_{i}$ denote the vertex that was toppled to get from $\mu_{i-1}$ to $\mu_{i}$, 
and let $m_{i}$ denote the mass that was toppled at this step. 
Let $t$ denote the number of moves necessary for this greedy algorithm to transport mass $p$ to distance at least $n$ from the root, 
i.e., $t = \min \left\{ i \geq 0 : \mu_{i} \left( B_{n} \right) \leq 1 - p \right\}$. 
Finally, for every $\ell \in \mathbb{N}$, let $t_{\ell}$ denote the number of moves necessary for this greedy algorithm to make the average level of the mass distribution at least $\ell$, 
i.e., $t_{\ell} := \min \left\{ i \geq 0 : M_{1} \left[ \mu_{i} \right] \geq \ell \right\}$. 

We first consider how the average level of the mass distribution changes with each toppling move. 
If $v_{i} = \rootrho$, then all the mass goes to the first level and hence we have that $M_{1} \left[ \mu_{i} \right] - M_{1} \left[ \mu_{i-1} \right] = m_{i}$. 
If $v_{i} \neq \rootrho$, then a $1/(d+1)$ fraction of the mass goes one level lower, while the rest of the mass goes one level higher, so 
$M_{1} \left[ \mu_{i} \right] - M_{1} \left[ \mu_{i-1} \right] = \tfrac{d-1}{d+1} m_{i}$. 
In every case we have that 
\[
M_{1} \left[ \mu_{i} \right] - M_{1} \left[ \mu_{i-1} \right] \geq \frac{d-1}{d+1} m_{i}. 
\]

Now fix $\ell < n$. By Lemma~\ref{lem:existence_of_large_mass}, for every $i < t_{\ell}$ we have that 
$m_{i} \geq d^{-(\ell+1)} / 4$. This implies that 
\[
M_{1} \left[ \mu_{t_{\ell}-1} \right] - M_{1} \left[ \mu_{t_{\ell-1}} \right] \geq \left( t_{\ell} - 1 - t_{\ell - 1} \right) \times \frac{d-1}{d+1} \times \frac{1}{4 d^{\ell + 1}}. 
\]
On the other hand, by the definition of $t_{\ell}$ we have that 
\[
M_{1} \left[ \mu_{t_{\ell}-1} \right] - M_{1} \left[ \mu_{t_{\ell-1}} \right] < \ell - \left( \ell - 1 \right) = 1. 
\]
Putting the previous displays together we obtain that 
\begin{equation}\label{eq:t_ell}
 t_{\ell} - t_{\ell-1} = O \left( d^{\ell} \right)
\end{equation}
for every $\ell < n$, where the implied constant depends only on $d$. Summing~\eqref{eq:t_ell} over $\ell$ from $1$ to $n-1$ we obtain that 
\[
t_{n-1} = O \left( d^{n} \right).
\]
Thus what remains is to show that 
$t - t_{n-1} = O \left( d^{n} \right)$. 
Recall that for every $i < t$ we have that 
$\mu_{i} \left( B_{n} \right) > 1 - p$. 
Since $\mathrm{Vol} \left( B_{n} \right) < d^n$, 
there must exist $v \in B_{n}$ such that 
$\mu_{i} \left( v \right) > \left( 1 - p \right) / d^{n}$. 
Hence for every $i \in (t_{n-1}, t]$ we have that $m_{i} > \left( 1 - p \right) / d^{n}$. 
Thus 
\[
 M_{1} \left[ \mu_{t} \right] - M_{1} \left[ \mu_{t_{n-1}} \right] > \left( t - t_{n-1} \right) \frac{d-1}{d+1} \left( 1 - p \right) / d^{n}. 
\]
On the other hand, since the support of $\mu_{t}$ is contained in $B_{n+1}$, we have that $M_{1} \left[ \mu_{t} \right] \leq n$, so 
\[
 M_{1} \left[ \mu_{t} \right] - M_{1} \left[ \mu_{t_{n-1}} \right] \leq n - \left( n - 1 \right) = 1. 
\]
Putting the previous two displays together we obtain that 
$t - t_{n-1} < \left( 1 - p \right)^{-1} \frac{d+1}{d-1} d^{n}$. 
\end{proof}

\subsection{A general bound for graphs where random walk has positive speed and entropy} \label{sec:proofs_pos_speed_general} 

In this subsection we present a general result for graphs where simple random walk has positive speed and entropy. 
Let $G = \left( V, E \right)$ be an infinite, connected, locally finite graph with $\rootv \in V$ a specified vertex, 
and let $\left\{ X_{t} \right\}_{t \geq 0}$ denote simple random walk on $G$ started from $\rootv$, i.e., with $X_{0} = \rootv$. 
We denote by $p_{t} \left( \cdot, \cdot \right)$ the $t$ step probability transition kernel. 
We start by introducing the basic notions of speed and entropy for random walk.

\begin{definition}\label{def:speed}
The (asymptotic) \emph{speed} of the random walk $\left\{ X_{t} \right\}_{t \geq 0}$ on $G$ is defined as 
\[
 \boldsymbol{\ell} := \lim_{t \to \infty} \frac{d \left( X_{0}, X_{t} \right)}{t}. 
\]
\end{definition}

Note that the triangle inequality implies subadditivity, that is, 
$d \left( X_{0}, X_{s+t} \right) \leq d \left( X_{0}, X_{s} \right) + d \left( X_{s}, X_{s+t} \right)$, 
and hence the speed of the random walk exists almost surely by Kingman's subadditive ergodic theorem (see, e.g.,~\cite[Theorem~14.44]{LP:book}). 

Recall that the entropy of a discrete random variable $X$ taking values in $\mathcal{X}$ is defined as 
\[
H \left( X \right)  = - \sum_{x \in \mathcal{X}} \p \left( X = x \right) \log \p \left( X = x \right),  
\] 
where in this paper we use $\log$ to denote the natural logarithm. 

\begin{definition}\label{def:entropy}
The \emph{asymptotic entropy}, also known as the \emph{Avez entropy}, of the random walk $\left\{ X_{t} \right\}_{t \geq 0}$ on $G$ is defined as 
\[
 \boldsymbol{h} := \lim_{t \to \infty} \frac{H \left( X_{t} \right)}{t}, 
\]
provided that this limit exists. 
\end{definition}

When $G$ is transitive, the sequence $\left\{ H \left( X_{t} \right) \right\}_{t \geq 0}$ is subadditive, and hence the Avez entropy exists by Fekete's lemma (see, e.g.,~\cite[Section~14.1]{LP:book}).

We recall two results concerning the asymptotic speed and the Avez entropy of the random walk. 
First, the positivity of these two quantities are related, as stated in the following theorem. 

\begin{theorem}\label{thm:pos_speed_entropy}[\cite{kaimanovich1983random},~\cite[Theorem~14.1]{LP:book}]
Let $G$ be a Cayley graph. 
Then the random walk has positive asymptotic speed, i.e., $\boldsymbol{\ell} > 0$, 
if and only if 
the Avez entropy of the random walk is positive, i.e., $\boldsymbol{h} > 0$. 
\end{theorem}

The following result is known as Shannon's theorem for random walks. 

\begin{theorem}\label{thm:Shannon}[\cite[Theorem~2.1]{kaimanovich1983random},~\cite[Theorem~14.10]{LP:book}]
Assume the setup described in the first paragraph of Section~\ref{sec:proofs_pos_speed_general} and in addition assume that $G$ is a transitive graph. Then we have that 
 \begin{equation*}
  \lim_{t \to \infty} \frac{1}{t} \log p_{t} \left( \rootv, X_{t} \right) = - \boldsymbol{h}
 \end{equation*}
almost surely. 
\end{theorem}

In the main result of this subsection, 
we provide sharp bounds in the exponent for the number of toppling moves necessary to transport mass $p$ to distance $n$ 
for graphs where simple random walk has positive asymptotic speed, positive Avez entropy, and which satisfy Shannon's theorem.

\begin{theorem}\label{thm:general_pos_speed}
Let $G = \left( V, E \right)$ be an infinite, connected, locally finite graph with $\rootv \in V$ a specified vertex, 
and let $\left\{ X_{t} \right\}_{t \geq 0}$ denote simple random walk on $G$ started from $\rootv$, i.e., with $X_{0} = \rootv$. 
Assume that the following three conditions hold: 
\begin{enumerate}
 \item\label{ass:speed} Simple random walk on $G$ has positive asymptotic speed, i.e., $\boldsymbol{\ell} > 0$. 
 \item\label{ass:entropy} Simple random walk on $G$ has positive Avez entropy, i.e., $\boldsymbol{h} > 0$. 
 \item\label{ass:conv} We have that 
 \begin{equation}\label{eq:conv_to_h}
  \lim_{t \to \infty} \frac{1}{t} \log p_{t} \left( \rootv, X_{t} \right) = - \boldsymbol{h}
 \end{equation}
 almost surely. 
\end{enumerate}
Then the minimum number of toppling moves needed to transport mass $p$ to distance at least $n$ from $\rootv$ is 
\begin{equation}\label{eq:general_pos_speed}
 N_{p} \left( G, B_{n}, \delta_{\rootv} \right) = \exp \left( n \times \frac{\boldsymbol{h}}{\boldsymbol{\ell}} \left( 1 + o \left( 1 \right) \right) \right).
\end{equation}
\end{theorem}

\begin{proof}
To prove the upper bound, 
we define a sequence of toppling moves that simulates the random walk, killed when it exits $B_{n}$, until time $t^{*} = \left( 1 + \eps \right) n / \boldsymbol{\ell}$, by which time most of the mass is outside of $B_{n}$. 
However, in order to get an upper bound of the correct order, we only do the toppling moves at the subset of sites that the random walk typically visits. 
The rest of the proof makes this precise. 

Fix $\eps > 0$ and let $t^{*} = \left( 1 + \eps \right) n / \boldsymbol{\ell}$. We first define the set of vertices on which we simulate the random walk. 
Let 
\[
 r_{n} := \max \left\{ r : \left| B_{r} \right| \leq n \right\}
\]
and note that $\lim_{n \to \infty} r_{n} = \infty$ due to the assumptions on $G$. 
Define also 
\begin{equation}\label{eq:Vtn}
V_{t,n} := \left\{ x \in B_{n} : \frac{1}{t} \log p_{t} \left( \rootv, x \right) \in \left( - \boldsymbol{h} \left( 1 + \eps \right), - \boldsymbol{h} \left( 1 - \eps \right) \right) \right\},
\end{equation}
and note that 
$\left| V_{t,n} \right| \leq \exp \left( t \boldsymbol{h} \left( 1 + \eps \right) \right)$ for every $t$, 
since $p_{t} \left( \rootv, x \right) \geq \exp \left( - t \boldsymbol{h} \left( 1 + \eps \right) \right)$ for every $x \in V_{t,n}$. 
Now define 
\[
 U_{n} := B_{r_{n}} \cup \bigcup_{t=r_{n}}^{t^{*}} V_{t,n}
\]
and let $Z_{t} := X_{t \wedge T_{U_{n}}}$ denote the random walk started at $\rootv$ and killed when it exits $U_{n}$. 
We can simulate the killed random walk 
$\left\{ Z_{t} \right\}_{t = 0}^{t^{*}}$ 
using $t^{*} \left| U_{n} \right|$ toppling moves. 
We shall show that 
\begin{equation}\label{eq:Zt*}
\p_{\rootv} \left( Z_{t^{*}} \notin B_{n} \right) \geq p
\end{equation}
if $n$ is large enough, 
which thus implies that 
\[
N_{p} \left( G, B_{n}, \delta_{\rootv} \right) 
\leq t^{*} \left| U_{n} \right| 
\leq t^{*} \left( n + t^{*} \exp \left( t^{*} \boldsymbol{h} \left( 1 + \eps \right) \right) \right)
\]
if $n$ is large enough. 
Since this holds for every $\eps > 0$, 
we get the desired upper bound stated in~\eqref{eq:general_pos_speed}. 

So what remains is to show~\eqref{eq:Zt*}. 
There are two ways that $Z_{t^{*}}$ can be in the ball $B_{n}$: 
either it is in the set $U_{n}$, 
or the random walk exited $U_{n}$ before exiting the ball $B_{n}$, 
and thus we have that 
\begin{equation}\label{eq:two_scenarios}
\p_{\rootv} \left( Z_{t^{*}} \in B_{n} \right) 
= 
\p_{\rootv} \left( Z_{t^{*}} \in U_{n} \right) 
+ \p_{\rootv} \left( Z_{t^{*}} \in B_{n} \setminus U_{n} \right).
\end{equation}
The first scenario is unlikely due to Assumption~\ref{ass:speed}. 
Specifically, if the killed random walk has not exited $U_{n}$, then its distance from $X_{0} = \rootv$ is less than $n$, so we have that 
\[
\p_{\rootv} \left( Z_{t^{*}} \in U_{n} \right) 
\leq 
\p_{\rootv} \left( d \left( X_{0}, X_{t^{*}} \right) < n \right) 
= 
\p_{\rootv} \left( \tfrac{1}{t^{*}} d \left( X_{0}, X_{t^{*}} \right) < \boldsymbol{\ell}/ \left( 1+\eps \right) \right).
\]
Assumption~\ref{ass:speed} implies that this latter probability goes to $0$, since $t^{*} \to \infty$ as $n \to \infty$. 
In particular, if $n$ is large enough then we have that 
$\p_{\rootv} \left( Z_{t^{*}} \in U_{n} \right) \leq (1-p)/2$. 
The second probability on the right hand side of~\eqref{eq:two_scenarios} is small due to Assumption~\ref{ass:conv}. 
First note that the random walk satisfies $Z_{t} \in U_{n}$ for all $t < r_{n}$ due to the construction of $U_{n}$. 
Now if the random walk exited $U_{n}$ before exiting $B_{n}$, then by the definition of $U_{n}$ there must exist a time 
$t \in \left\{ r_{n}, r_{n} + 1, \dots, t^{*} \right\}$ 
such that 
$X_{t} \in B_{n} \setminus V_{t,n}$. 
This implies that 
\[
\p_{\rootv} \left( Z_{t^{*}} \in B_{n} \setminus U_{n} \right)
\leq 
\p_{\rootv} \left( \exists\, t \geq r_{n} : \tfrac{1}{t} \log p_{t} \left( X_{0}, X_{t} \right) \notin \left( - \boldsymbol{h} \left( 1 + \eps \right), - \boldsymbol{h} \left( 1 - \eps \right) \right) \right).
\]
Assumption~\ref{ass:conv} implies that this latter probability converges to $0$ as $r_{n} \to \infty$. 
Since $r_{n} \to \infty$ as $n \to \infty$, we have in particular that 
$\p_{\rootv} \left( Z_{t^{*}} \in B_{n} \setminus U_{n} \right) \leq (1-p)/2$ 
if $n$ is large enough. 
This concludes the proof of~\eqref{eq:Zt*}.

To prove the lower bound stated in~\eqref{eq:general_pos_speed}, we again start by smoothing the initial mass distribution, by simulating simple random walk on $G$ until time $t^{**} := \left( 1 - \eps \right) n / \boldsymbol{\ell}$. 
As we shall see, the mass distribution is then approximately uniformly distributed on a subset of $B_{n}$ of size approximately 
$\exp \left( t^{**} \boldsymbol{h} \right)$. 
In order to transport a constant mass outside of $B_{n}$, it is then necessary to topple the mass at a constant fraction of the vertices in this subset, which leads to the desired lower bound. 
The rest of the proof makes this precise. 

Fix $\eps > 0$ and let $t^{**} := \left( 1 - \eps \right) n / \boldsymbol{\ell}$. 
The choice of $t^{**}$ is due to the fact that, by Assumption~\ref{ass:speed}, with probability close to $1$, simple random walk on $G$ does not exit the ball $B_{n}$ until time $t^{**}$. 
Let $Z'_{t} := X_{t \wedge T_{B_{n}}}$ denote the simple random walk on $G$ killed when it exits $B_{n}$. 
Starting with the initial mass distribution $\delta_{\rootv}$, 
we apply a sequence of $t^{**} \times \mathrm{Vol} \left( B_{n} \right)$ toppling moves that simulate $t^{**}$ steps of the killed random walk $\left\{ Z'_{t} \right\}_{t = 0}^{t^{**}}$, 
to arrive at a new mass distribution $\mu$. 
By Corollary~\ref{cor:smoothing} we have that 
$N_{p} \left( G, B_{n}, \delta_{\rootv} \right) \geq N_{p} \left( G, B_{n}, \mu \right)$, 
so it suffices to bound from below this latter quantity. 
Recall the definition of $V_{t,n}$ from~\eqref{eq:Vtn}. 
By the definition of $t^{**}$ and Assumptions~\ref{ass:speed} and~\ref{ass:conv}, it follows that 
\[
 \mu \left( V_{t^{**},n} \right) \geq 1 - \frac{p}{2}
\]
if $n$ is large enough. 
Therefore, in order to transport mass $p$ outside of $B_{n}$ starting from the mass distribution $\mu$, it is necessary to transport mass at least $p/2$ from vertices in $V_{t^{**},n}$. 
However, 
$\mu \left( x \right) \leq \exp \left( - t^{**} \boldsymbol{h} \left( 1 - \eps \right) \right)$ 
for every $x \in V_{t^{**},n}$, 
so at least 
\[
\frac{p}{2} \times \exp \left( t^{**} \boldsymbol{h} \left( 1 - \eps \right) \right) 
= \frac{p}{2} \times \exp \left( n \times \frac{\boldsymbol{h}}{\boldsymbol{\ell}} \left( 1 - \eps \right)^{2} \right) 
\]
vertices in $V_{t^{**},n}$ need to be toppled at least once. 
Since this holds for any $\eps > 0$, the result follows. 
\end{proof}

\subsection{Galton-Watson trees} \label{sec:proofs_gw_trees} 

The behavior of simple random walk on Galton-Watson trees was studied in great detail by Lyons, Pemantle, and Peres~\cite{lpp}. 
Using their results, combined with the general results of Section~\ref{sec:proofs_pos_speed_general}, we can prove Theorem~\ref{thm:GW}. 

Specifically, Lyons, Pemantle, and Peres~\cite{lpp} showed that the three conditions of Theorem~\ref{thm:general_pos_speed} hold for almost every Galton-Watson tree. 
Furthermore, they also show that the ratio of the asymptotic entropy and speed is equal to the Hausdorff dimension of harmonic measure on the boundary of a Galton-Watson tree. 
Here we state the basic results necessary to conclude Theorem~\ref{thm:general_pos_speed}, and refer to~\cite{lpp} for much more detailed results, 
including formulas for the asymptotic speed and entropy as a function of the offspring distribution of the Galton-Watson branching process. 
We state this result for nondegenerate offspring distributions, as degenerate offspring distributions (giving rise to $m$-ary trees) are treated more carefully in Section~\ref{sec:proofs_dary_trees}. 

\begin{theorem}\label{thm:LPP_GWT}[\cite[Theorem~1.1,~Theorem~3.2,~Theorem~9.7]{lpp}]
Fix a nondegenerate offspring distribution with mean $m > 1$ and let $\mathrm{GWT}$ be a Galton-Watson tree obtained with this offspring distribution, on the event of nonextinction. 
Let $\left\{ X_{t} \right\}_{t \geq 0}$ denote simple random walk on $\mathrm{GWT}$ started from the root $\rootrho$, i.e., with $X_{0} = \rootrho$, 
and let $p_{t} \left( \cdot, \cdot \right)$ denote the $t$ step probability transition kernel. 
For almost every Galton-Watson tree $\mathrm{GWT}$ the following statements hold. 
The asymptotic speed $\boldsymbol{\ell}$ and Avez entropy $\boldsymbol{h}$ of the random walk exist and are positive almost surely. 
Moreover, we have that 
\[
 \frac{\boldsymbol{\ell}}{\boldsymbol{h}} = \mathbf{dim} 
\]
almost surely, where $\mathbf{dim}$ is the dimension of harmonic measure, which is almost surely a constant less than $\log m$. 
Furthermore, we have that 
\begin{equation*}
  \lim_{t \to \infty} \frac{1}{t} \log p_{t} \left( \rootv, X_{t} \right) = - \boldsymbol{h}
\end{equation*}
almost surely. 
\end{theorem}

\begin{proof}[Proof of Theorem~\ref{thm:GW}]
Theorem~\ref{thm:LPP_GWT} shows that the three conditions of Theorem~\ref{thm:general_pos_speed} hold for almost every Galton-Watson tree. 
Hence Theorem~\ref{thm:GW} follows from Theorem~\ref{thm:general_pos_speed}. 
\end{proof}

\subsection{Product of trees} \label{sec:proofs_product_trees} 

In this subsection we apply the general result derived in Section~\ref{sec:proofs_pos_speed_general} to obtain tight bounds for the specific case of the product of trees. 
As we shall see, the key observation is that random walk typically does not visit the entire ball $B_{n}$ on the product of trees, 
due to its different speeds on the edges belonging to different trees. 

Let $\mathbb{T}_{d}$ denote the $(d+1)$-regular tree.\footnote{In Section~\ref{sec:proofs_dary_trees}, $\mathbb{T}_{d}$ denotes the $d$-ary tree, which differs from the $(d+1)$-regular tree in that the root $\rootrho$ has degree $d$ instead of $d+1$. This difference is not important for the questions we consider, so we allow ourselves this abuse of notation.} 
We define the Cartesian product 
$\mathbb{T}_{d} \times \mathbb{T}_{k}$ 
to have vertex set 
$V \left( \mathbb{T}_{d} \times \mathbb{T}_{k} \right) = V \left( \mathbb{T}_{d} \right) \times V \left( \mathbb{T}_{k} \right)$ 
and edge set defined as follows: 
\[
\left( u, v \right) \sim \left( u', v' \right) 
\Longleftrightarrow 
\begin{cases}  
u \sim u' \text{ and } v=v', & \text{or} \\
u = u' \text{ and } v \sim v'.
\end{cases}
\]
Note that 
$\mathbb{T}_{d} \times \mathbb{T}_{k}$ 
is a $(d+k+2)$-regular graph. 
Note also that $\mathbb{T}_{1}$ is isomorphic to $\mathbb{Z}$, and so $T_{1} \times T_{1}$ is isomorphic to $\mathbb{Z}^{2}$; 
this graph is covered by Theorem~\ref{thm:zd}, and hence we may assume that $d + k \geq 3$. 

\begin{proof}[Proof of Theorem~\ref{thm:product_trees}]
We prove this result by appealing to the general result of Theorem~\ref{thm:general_pos_speed}. 
Therefore we need to check that the three assumptions of Theorem~\ref{thm:general_pos_speed} hold 
and we also need to compute the asymptotic speed $\boldsymbol{\ell}$ and the Avez entropy $\boldsymbol{h}$ for simple random walk on $\mathbb{T}_{d} \times \mathbb{T}_{k}$.

Let $\left\{ X_{t} \right\}_{t \geq 0}$ denote simple random walk on $\mathbb{T}_{d} \times \mathbb{T}_{k}$ with $X_{0} = \rootrho$. 
We start by computing the speed of random walk. 
Recall that the speed of random walk on the $(d+1)$-regular tree $\mathbb{T}_{d}$ is 
$\tfrac{d-1}{d+1}$. 
Moreover, the probability of random walk on $\mathbb{T}_{d} \times \mathbb{T}_{k}$ making a step in the first coordinate (corresponding to $\mathbb{T}_{d}$) is 
$\tfrac{d+1}{d+k+2}$. 
Hence the speed of random walk $\left\{ X_{t} \right\}_{t \geq 0}$ is the convex combination of the speeds of random walk on the regular trees $\mathbb{T}_{d}$ and $\mathbb{T}_{k}$: 
\begin{equation}\label{eq:product_trees_speed}
\boldsymbol{\ell} = \frac{d+1}{d+k+2} \times \frac{d-1}{d+1} + \frac{k+1}{d+k+2} \times \frac{k-1}{k+1} = \frac{d+k-2}{d+k+2}. 
\end{equation}
Since $d+k \geq 3$, the speed is positive: $\boldsymbol{\ell} > 0$. 

Since $\mathbb{T}_{d} \times \mathbb{T}_{k}$ is a transitive graph, we know from Theorem~\ref{thm:Shannon} that~\eqref{eq:conv_to_h} holds. 
Thus what remains is to compute the Avez entropy of $\left\{ X_{t} \right\}_{t \geq 0}$ and to show that it is positive. 
We start by computing the Avez entropy of random walk on $\mathbb{T}_{d}$. 
Let $\left\{ Y_{t} \right\}_{t \geq 0}$ denote simple random walk on $\mathbb{T}_{d}$ started from the root, i.e., with $Y_{0} = \rootrho$, 
and let $\left| Y_{t} \right|$ denote the distance of $Y_{t}$ from the root $\rootrho$. 
By the chain rule of conditional entropy we have that 
\[
H \left( Y_{t} \right) = H \left( \left| Y_{t} \right| \right) + H \left( Y_{t} \, \middle| \, \left| Y_{t} \right| \right).
\]
Since $\left| Y_{t} \right|$ takes values in $\left\{ 0, 1, \dots, t \right\}$, we have that 
$H \left( \left| Y_{t} \right| \right) \in \left[ 0, \log \left( t + 1 \right) \right]$. 
For $i \in \left[ t \right]$, conditioned on $\left| Y_{t} \right| = i$, the random variable $Y_{t}$ is uniformly distributed among all $\left( d + 1 \right) d^{i-1}$ vertices at distance $i$ from the root. 
Hence, using the fact that the asymptotic speed of $\left\{ Y_{t} \right\}_{t \geq 0}$ is $\tfrac{d-1}{d+1}$, we have that 
\begin{align*}
H \left( Y_{t} \, \middle| \, \left| Y_{t} \right| \right) 
&= \sum_{i=1}^{t} \p \left( \left| Y_{t} \right| = i \right) \times \log \left( \left( d + 1 \right) d^{i-1} \right) 
= \log \left( 1 + 1/d \right) \times \p \left( \left| Y_{t} \right| \neq 0 \right) + \log \left( d \right) \times \E \left[ \left| Y_{t} \right| \right] \\
&= \log \left( d \right) \times \frac{d-1}{d+1} \times t \left( 1 + o \left( 1 \right) \right).
\end{align*}
We conclude that the Avez entropy of $\left\{ Y_{t} \right\}_{t \geq 0}$ is 
\[
\boldsymbol{h}_{Y} = \log \left( d \right) \times \frac{d-1}{d+1}.
\]
Now let $\left\{ Z_{t} \right\}_{t \geq 0}$ denote simple random walk on $\mathbb{T}_{k}$ started from the root, i.e., with $Z_{0} = \rootrho$, 
and let $\left\{ Y_{t} \right\}_{t \geq 0}$ and $\left\{ Z_{t} \right\}_{t \geq 0}$ be independent. 
Furthermore, independently of everything else, let $\left\{ W_{i} \right\}_{i \geq 1}$ be i.i.d.\ Bernoulli random variables with expectation $\tfrac{d+1}{d+k+2}$, 
and let $S_{t} := \sum_{i=1}^{t} W_{i}$. 
Then, by construction, 
$\left\{ \left( Y_{S_{t}}, Z_{t - S_{t}} \right) \right\}_{t \geq 0}$ 
has the same distribution as 
$\left\{ X_{t} \right\}_{t \geq 0}$. 
We can again use the chain rule of conditional entropy, this time conditioning on $S_{t}$, to get that 
\[
H \left( X_{t} \right) = H \left( S_{t} \right) + H \left( \left( Y_{S_{t}}, Z_{t - S_{t}} \right) \, \middle| \, S_{t} \right). 
\]
Since $S_{t}$ takes values in $\left\{0, 1, \dots, t \right\}$, we have that 
$H \left( S_{t} \right) \in \left[ 0, \log \left( t + 1 \right) \right]$. 
Conditioning on $S_{t}$, the random variables $Y_{S_{t}}$ and $Z_{t-S_{t}}$ are independent, and hence 
$H \left( \left( Y_{S_{t}}, Z_{t - S_{t}} \right) \, \middle| \, S_{t} \right) 
= H \left( Y_{S_{t}} \, \middle| \, S_{t} \right) 
+ H \left( Z_{t - S_{t}} \, \middle| \, S_{t} \right)$. 
Therefore, using the computation from above of the entropy of random walk on a regular tree, together with the fact that 
$S_{t} = \tfrac{d+1}{d+k+2} t \left( 1 + o \left( 1 \right) \right)$ with high probability, 
we obtain that the Avez entropy of $\left\{ X_{t} \right\}_{t \geq 0}$ is 
\begin{align*}
\boldsymbol{h}_{X} 
&= \frac{d+1}{d+k+2} \boldsymbol{h}_{Y} + \frac{k+1}{d+k+2} \boldsymbol{h}_{Z} \\
&= \frac{d-1}{d+k+2} \log \left( d \right) + \frac{k-1}{d+k+2} \log \left( k \right). 
\end{align*}
Since at least one of $d$ and $k$ is greater than $1$, the Avez entropy $\boldsymbol{h}_{X}$ is positive. 
Plugging in the values of $\boldsymbol{\ell}$ and $\boldsymbol{h}$ into the conclusion of Theorem~\ref{thm:general_pos_speed}, we obtain the desired result. 
\end{proof}

\section{Graphs with bounded degree and exponential decay of the Green's function} \label{sec:proofs_bdd_deg_Green} 

In this section we study graphs of bounded degree with exponential decay of the Green's function, 
showing that the minimum number of toppling moves necessary to transport a constant mass to distance at least $n$ 
is exponential in $n$. 

Let $G = \left( V, E \right)$ be an infinite and connected graph with bounded degree. 
Recall the definition of the Green's function $g$ for simple random walk on $G$ from Definition~\ref{def:green}. 
We say that the Green's function has \emph{exponential decay} if there exist positive and finite constants $a$ and $a'$ depending only on $G$ such that 
\begin{equation}\label{eq:exp_decay_Green}
g \left( x, y \right) \leq \exp \left( - a \times d \left( x, y \right) + a' \right)
\end{equation}
for every $x, y \in V$, where $d$ denotes graph distance. 
Note that the Green's function cannot decay faster than exponentially as a function of the distance.

If simple random walk on $G$ has positive speed and positive entropy, then the Green's function has exponential decay (see~\cite{bp,bhm}). 
However, the reverse implication does not hold, 
and hence the method described in Section~\ref{sec:proofs_pos_speed_general} 
to bound the minimum number of toppling moves $N_{p} \left( G, B_{n}, \delta_{\rooto} \right)$ 
does not work in general. 
As an example, we shall investigate the lamplighter graph with base graph $\mathbb{Z}$, 
for which it has been shown that the speed and entropy of simple random walk are both zero (see~\cite[Proposition~6.2]{kaimanovich1983random}).

We restate Theorem~\ref{thm:bdd_deg_Green} more precisely before proving it.

\begin{theorem}\label{thm:bdd_deg_Green_restatement}
Let $G = \left( V, E \right)$ be an infinite and connected graph such that 
every vertex has degree at most $D$ and 
the Green's function $g$ for simple random walk on $G$ satisfies~\eqref{eq:exp_decay_Green}. 
Start with initial unit mass $\delta_{\rootv}$ at a vertex $\rootv \in V$ and let $p \in (0,1)$ be constant. 
The minimum number of toppling moves needed to transport mass $p$ to distance at least $n$ from $\rootv$ is 
\[
N_{p} \left( G, B_{n}, \delta_{\rootv} \right) = \exp \left( \Theta \left( n \right) \right),
\]
where the implied constants depend only on $p$, $D$, $a$, and $a'$. 
\end{theorem}

\begin{proof}
For the upper bound we use the general bound given by Theorem~\ref{thm:volumeexittime}. 
Since $G$ has bounded degree, 
the volume of a ball grows at most exponentially: 
$\mathrm{Vol} \left( B_{n} \right) \leq \sum_{i=0}^{n-1} D^{i} \leq D^{n}$. 
Furthermore, the exit time of random walk from a ball can also be bounded, e.g., in the following crude way. 
The exit time $T_{B_{n}}$ is equal to the number of visits to vertices in $B_{n}$ before the random walk exits $B_{n}$, 
and hence can be bounded by the total number of visits to vertices in $B_{n}$. 
Thus we obtain the following crude bound: 
$\E_{\rootv} \left[ T_{B_{n}} \right] 
\leq 
\sum_{x \in B_{n}} g \left( \rootv, x \right) 
\leq 
e^{a'} \mathrm{Vol} \left( B_{n} \right)$. 
Hence using Theorem~\ref{thm:volumeexittime} we have that 
\[
N_{p} \left( G, B_{n}, \delta_{\rootv} \right) 
\leq 
\left( 1 - p \right)^{-1} e^{a'} D^{2n}.
\]

For the lower bound we again perform smoothing of the initial mass distribution.  
Let $\left\{ X_{t} \right\}_{t \geq 0}$ denote simple random walk on $G$ with $X_{0} = \rootv$, 
and let 
$Z_{t} := X_{t \wedge T_{B_{n-1}}}$ denote the random walk killed when it exits the ball $B_{n-1}$. 
Let 
$t^{*}$ be such that 
\begin{equation}\label{eq:mass_inside}
\p_{\rootv} \left( Z_{t^{*}} \in B_{n-1} \right) \leq p/2. 
\end{equation}
Starting with the initial mass distribution $\delta_{\rootv}$, 
we apply a sequence of $t^{*} \times \mathrm{Vol} \left( B_{n} \right)$ toppling moves 
that simulate $t^{*}$ steps of the killed random walk $\left\{ Z_{t} \right\}_{t \geq 0}$, 
to arrive at a new mass distribution $\mu$. 
By Corollary~\ref{cor:smoothing} we have that 
$N_{p} \left( G, B_{n}, \delta_{\rootv} \right) 
\geq 
N_{p} \left( G, B_{n}, \mu \right)$, 
so it suffices to bound from below this latter quantity. 

Denote the boundary of $B_{n-1}$ by
$\partial B_{n-1} := \left\{ x \in V : d \left( \rootv, x \right) = n - 1 \right\}$. 
For every $x \in \partial B_{n-1}$  
we can bound the mass at $x$ using the Green's function: 
\[
\mu \left( x \right) = \p_{\rootv} \left( Z_{t^{*}} = x \right) 
\leq \p_{\rootv} \left( X_{T_{B_{n-1}}} = x \right) 
\leq \sum_{k=0}^{\infty} \p_{\rootv} \left( X_{k} = x \right) 
= g \left( \rootv, x \right) 
\leq \exp \left( - a n + a + a' \right),
\]
where in the last inequality we used~\ref{eq:exp_decay_Green}. 
Now~\eqref{eq:mass_inside} implies that 
$\mu \left( \partial B_{n-1} \right) \geq 1 - p/2$, 
and so in order to transport mass at least $p$ to outside of $B_{n}$ starting from $\mu$, 
it is necessary to transport mass at least $p/2$ from the vertices in $\partial B_{n-1}$. 
However, the display above shows that every $x \in \partial B_{n-1}$ has mass at most 
$\exp \left( - a n + a + a' \right)$, 
so this requires at least $(p/2) \times \exp \left( a n - a - a' \right)$ toppling moves. 
Hence 
\[
N_{p} \left( G, B_{n}, \mu \right) \geq \frac{p}{2 \exp \left( a+a' \right)} \times e^{a n}. \qedhere 
\]
\end{proof}

\subsection{The lamplighter graph} \label{sec:proofs_lamplighter} 

We illustrate the results above with the lamplighter graph, which is an example of a graph with bounded degree and exponential decay of the Green's function. 

\begin{definition}\label{def:lamplighter}
The \emph{lamplighter group} is the wreath product $\mathbb{Z}_{2} \wr \mathbb{Z}$. 
The elements of the group are pairs of the form 
$\left( \eta, y \right)$, 
where $\eta : \mathbb{Z} \to \mathbb{Z}_{2}$ 
and $y \in \mathbb{Z}$. The group operation is 
\[
 \left( \eta_{1}, y_{1} \right) \left( \eta_{2}, y_{2} \right) := \left( \eta, y_{1} + y_{2} \right), 
\]
where 
$\eta \left( x \right) = \eta_{1} \left( x \right) + \eta_{2} \left( x - y_{1} \right) \mod 2$. 
\end{definition}
The reason for the name is that we may think of a lamp being present at each vertex of $\mathbb{Z}$, 
with a lamplighter walking on $\mathbb{Z}$ and turning lights on and off. 
A group element $\left( \eta, y \right)$ corresponds to the on/off configuration of the lamps $\eta$ and the position of the lamplighter $y$. 
Multiplying with the group elements $\left( \mathbf{0}, 1 \right)$ and $\left( \mathbf{0}, -1 \right)$ corresponds to the lamplighter moving to the right or to the left, 
and multiplying with $\left( \mathbf{1}_{0}, 0 \right)$ corresponds to flipping the light at the position of the lamplighter. 
Consider the random walk on the lamplighter group associated with the measure $\nu \ast \mu \ast \nu$, 
where $\mu$ is a simple random walk step by the lamplighter, 
and $\nu$ is a measure causing the lamplighter to randomize the current lamp. 
That is, 
$\mu \left( \mathbf{0}, \pm 1 \right) = 1/2$ 
and 
$\nu \left( \mathbf{1}_{0}, 0 \right) = \nu \left( \mathbf{0}, 0 \right) = 1/2$. 
In words, each step of the random walk corresponds to a ``randomize-move-randomize'' triple. 
We call the graph corresponding to this random walk the lamplighter graph and denote it by $\mathcal{G}$. 
The transition probabilities for this random walk have been well studied, which allow us to conclude the following result.

\begin{theorem}\label{thm:lamplighter}
Let $\rooto$ denote the identity element of the lamplighter group $\mathbb{Z}_{2} \wr \mathbb{Z}$, 
start with initial unit mass $\delta_{\rooto}$ at $\rooto$, and let $p \in \left( 0, 1 \right)$ be constant. 
The minimum number of toppling moves needed to transport mass $p$ to distance at least $n$ from $\rooto$ is 
\[
N_{p} \left( \mathcal{G}, B_{n}, \delta_{\rooto} \right) = \exp \left( \Theta \left( n \right) \right).
\]
\end{theorem}

\begin{proof}
In order to apply Theorem~\ref{thm:bdd_deg_Green_restatement} we need to check that the two conditions of the theorem hold. 
First, $\mathcal{G}$ is $8$-regular, so the first condition holds. 
The fact that the Green's function decays exponentially follows directly from~\cite[Theorems~1 and~2]{revelle}. 
\end{proof}

\section{Open problems} \label{sec:open} 

\begin{itemize}
 \item \textbf{Connections to maximum overhang problems.} Paterson~et~al.~\cite{overhang} studied the controlled diffusion problem on $\mathbb{Z}$ due to its connections with the maximum overhang problem in one dimension: 
 how far can a stack of $n$ identical blocks be made to hang over the edge of a table? 
 
 The answer was widely believed to be of order $\log(n)$, 
 by considering harmonic stacks in which $n$ unit length blocks are placed one on top of the other, 
 with the $i^{\text{th}}$ block from the top extending by $1/(2i)$ beyond the block below it. 
 This construction has an overhang of $\sum_{i=1}^{n} 1/(2i) \sim \tfrac{1}{2} \ln \left( n \right)$. 
 
 However, Paterson and Zwick showed that this belief is false, by constructing an example with overhang on the order of $n^{1/3}$~\cite{paterson2006overhang}. 
 Subsequently, Paterson~et~al.\ showed that this is best possible up to a constant factor~\cite{overhang}. 
 The authors proved this result by connecting the overhang problem to the controlled diffusion problem on $\mathbb{Z}$. 
 
 This connection naturally leads to the following question: are the results presented in this paper relevant for maximum overhang problems in higher dimensions?
 
 \item \textbf{Effectiveness of the greedy algorithm.} Under what circumstances is the greedy algorithm (approximately) optimal? 
 
 \item \textbf{Small mass asymptotics.} What is the dependence of $N_{p} \left( G, B_{n}, \rootv \right)$ on $p$ as $p \to 0$? 
\end{itemize}


\section*{Acknowledgements}

We thank Peter Winkler for helpful discussions and Jordan Greenblatt for comments on an early draft. LF acknowledges the support of a MacCracken Fellowship from NYU and the Theory group at Microsoft Research Redmond for its hospitality.


\bibliographystyle{plain}
\bibliography{overhanghd}




\end{document}